\documentclass{amsart}

\usepackage{microtype}

\usepackage{amsmath,amssymb,amsthm}
\usepackage{hyperref}
\usepackage{mathrsfs}
\usepackage{stmaryrd}
\usepackage{tikz-cd,tikz-qtree}
\usepackage{enumerate}
\usepackage{soul,color}
\usepackage{quiver}
\usepackage{wrapfig}
\usepackage{tikzit}

\tikzstyle{node}=[fill=white, draw=black, shape=circle, minimum size=7mm]

\tikzstyle{edge}=[->]

\newcommand{\xrightsquigarrow}[1]{\overset{#1}{\rightsquigarrow}}

\newcommand{\tot}{\xrightarrow}
\newcommand{\tp}{\rightsquigarrow}

\newcommand{\abv}{\uparrow}
\newcommand{\blw}{\downarrow}

\newcommand{\F}{\ensuremath{F}}
\newcommand{\U}{\ensuremath{U}}
\newcommand{\E}{\ensuremath{E}}

\newcommand{\edgel}{\,\triangleleft\,}
\newcommand{\lexdft}{<_D}
\newcommand{\lexbft}{<_B}

\newtheorem{theorem}{Theorem}
\newtheorem{corollary}[theorem]{Corollary}
\newtheorem{lemma}[theorem]{Lemma}
\theoremstyle{definition}
\newtheorem{definition}{Definition}
\theoremstyle{remark}
\newtheorem*{remark}{Remark}

\calclayout

\makeatletter
\newcommand\thankssymb[1]{\textsuperscript{\@fnsymbol{#1}}}
\makeatother

\newif\ifarxiv
\arxivtrue

\newcommand{\siddhuack}{Supported by \emph{DFF $\mid$ Natural Sciences} Research Project 1 115554 \emph{Cons-Free Recursion Theory}.}
\newcommand{\robinack}{Supported by \emph{DFF $\mid$ Natural Sciences} International Postdoctoral Fellowship 0131-00025B \emph{Landauer Meets von Neumann: Reversibility in Categorical Quantum Semantics}.}

\title{Graph Traversals as Universal Constructions}

\author[S. Bhaskar]{Siddharth Bhaskar\thankssymb{1}}
\address{University of Copenhagen, Denmark}
\email{sbhaskar@di.ku.dk}
\thanks{\thankssymb{1}\siddhuack}

\author[R. Kaarsgaard]{Robin Kaarsgaard\thankssymb{2}}
\address{University of Edinburgh, United Kingdom}
\email{robin.kaarsgaard@ed.ac.uk}
\thanks{\thankssymb{2}\robinack}

\begin{document}

\maketitle

\begin{abstract}
We exploit a decomposition of graph traversals to give a novel characterization of depth-first \emph{and} breadth-first traversals as universal constructions. Specifically, we introduce functors from two different categories of \emph{edge-ordered} directed graphs into two different categories of transitively closed edge-ordered graphs; one defines the lexicographic depth-first traversal and the other the lexicographic breadth-first traversal. We show that each functor factors as a composition of universal constructions, and that the usual presentation of traversals as linear orders on vertices can be recovered with the addition of an inclusion functor. Finally, we raise the question of to what extent we can recover search algorithms from the categorical description of the traversal they compute.
\end{abstract}

\section{Introduction}

Graph searches are algorithms for visiting the vertices in a connected graph from a prescribed source. Both graph searches and their resulting vertex orders, or \emph{traversals}, are absolutely fundamental in the theory of graph algorithms, and have important applications in other areas of theoretical computer science such as computational complexity theory.

We start from the premise that a concept as natural as a traversal should be obtainable canonically from the original graph. For example, let us consider the graph $G$ in Figure \ref{fig:4-cycle}, and fix $a$ as a source. Of the six vertex orderings of $G$ starting with $a$, two are not traversals: $(a,d,b,c)$ and $(a,d,c,b)$. This is because while searching a graph, each vertex added must be in the boundary of previously visited vertices, but $d$ is not in the boundary of $\{a\}$.

\ifarxiv
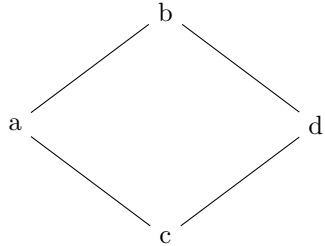
\begin{figure}[b]
    \centering
    \begin{tikzpicture}
              \node (a) at (0,0) {a};
              \node (b) at (2,1.5) {b};
              \node (c) at (2,-1.5) {c};
              \node (d) at (4,0) {d};
              
              \draw (a) edge (b);
              \draw (a) edge (c);
              \draw (b) edge (d);
              \draw (c) edge (d);
        \end{tikzpicture}
    \caption{A 4-cycle}
    \label{fig:4-cycle}
\end{figure}
\else
\begin{figure}
    \centering
    \begin{tikzpicture}
              \node (a) at (0,0) {a};
              \node (b) at (2,1.5) {b};
              \node (c) at (2,-1.5) {c};
              \node (d) at (4,0) {d};
              
              \draw (a) edge (b);
              \draw (a) edge (c);
              \draw (b) edge (d);
              \draw (c) edge (d);
        \end{tikzpicture}
    \caption{A 4-cycle}
    \label{fig:4-cycle}
\end{figure}
\fi

Of the four remaining vertex orders, two are breadth-first traversals ($(a,b,c,d)$ and $(a,c,b,d)$) and two are depth-first traversals ($(a,b,d,c)$ and $(a,c,d,b)$). This can easily be checked by hand: in a breadth-first search, we go level-by-level, and must visit both $b$ and $c$ before we visit $d$. In a depth-first search, we prioritize the neighbor of the most recently visited vertex, so we visit $d$ before the latter of $\{b,c\}$.

However, notice that there is no way of canonically distinguishing between the two breadth-first or the two depth-first traversals. Concretely, once we visit $a$, there is no canonical way to choose between $b$ and $c$. A natural fix is to linearly order each neighborhood and visit lesser neighbors first. We call the resulting traversals \emph{lexicographic}. For example, if we say that $b < c$, then the lexicographic breadth-first traversal is $(a,b,c,d)$ and the lexicographic depth-first traversal is $(a,b,d,c)$. If we say that $c < b$, we get $(a,c,b,d)$ and $(a,c,d,b)$ respectively. We call a graph whose neighborhoods are linearly ordered an \emph{edge-ordered} graph.

In the present paper, we show that both the lexicographic breadth-first traversal and lexicographic depth-first traversal are canonically obtainable from a given edge-ordered graph with a distinguished source. Specifically, we equip edge-ordered graphs with two different kinds of morphisms, and obtain lexicographic breadth- and depth-first traversals by applying a functor out of each category of edge-ordered graphs into the category of linear orders. We furthermore factor each functor as a composition of a forgetful functor and a sequence of \emph{universal} (free and cofree) constructions on edge-ordered graphs.

At a first approximation, each lexicographic traversal can be expressed as the composition of a least-path tree and a transitive closure (see Figure~\ref{fig:traversal}). This decomposition was first observed in \cite{TK95}---and not in the context of category theory--- and used to derive efficient parallel algorithms; see also \cite{ACD20,TK01}. The main technical contribution of our paper is in identifying precisely the right notions of edge-ordered graphs and homomorphisms that allow us to formulate least-path trees and transitive closures as universal constructions.


To the best of our knowledge, equipping algorithmic problems with a categorical structure is a relatively recent idea. While graphs have been studied extensively from a categorical point of view, the focus has been on topics such as graph rewriting and string diagrams~\cite{DixonKissinger2013,LackSob2005} and relationships with properads~\cite{Kock2016} rather than graph algorithms. The only other papers that we know of in the algorithmic vein are that of Master on the open algebraic path problem \cite{Mas20} and the compositional algorithm by Rathke, Soboci{\'{n}}ski and Stephens~\cite{RathSobSte2014} for reachability in Petri nets. These papers suggest that there might be a categorical setting for reachability problems on graphs, generally understood.

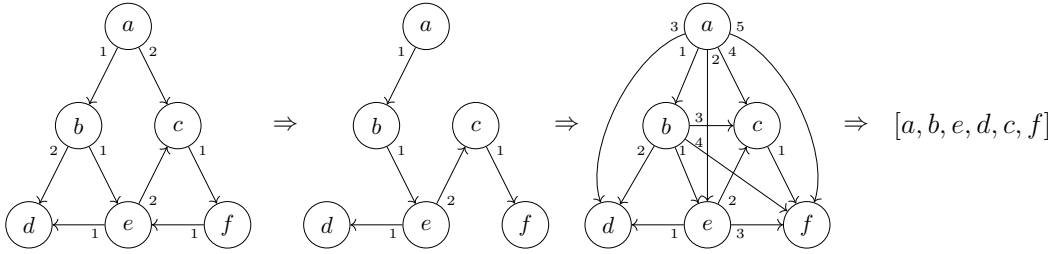
\begin{figure}
    \centering
    \scalebox{0.88}{\begin{tikzpicture}
	\begin{pgfonlayer}{nodelayer}
		\node [style=node] (0) at (-12.5, 3) {$a$};
		\node [style=node] (1) at (-14, 0) {$b$};
		\node [style=node] (2) at (-11, 0) {$c$};
		\node [style=node] (3) at (-12.5, -3) {$e$};
		\node [style=node] (4) at (-15.5, -3) {$d$};
		\node [style=node] (5) at (-9.5, -3) {$f$};
		\node [style=none] (6) at (-13.25, 2.25) {\tiny $1$};
		\node [style=none] (7) at (-11.75, 2.25) {\tiny $2$};
		\node [style=none] (8) at (-14.75, -0.75) {\tiny $2$};
		\node [style=none] (9) at (-13.25, -0.75) {\tiny $1$};
		\node [style=none] (10) at (-13.5, -3.25) {\tiny $1$};
		\node [style=none] (11) at (-11.75, -2.25) {\tiny $2$};
		\node [style=none] (12) at (-10.25, -0.75) {\tiny $1$};
		\node [style=none] (13) at (-10.5, -3.25) {\tiny $1$};
		\node [style=none] (14) at (-7.75, 0) {\large $\Rightarrow$};
		\node [style=node] (15) at (-3.5, 3) {$a$};
		\node [style=node] (16) at (-5, 0) {$b$};
		\node [style=node] (17) at (-2, 0) {$c$};
		\node [style=node] (18) at (-3.5, -3) {$e$};
		\node [style=node] (19) at (-6.5, -3) {$d$};
		\node [style=node] (20) at (-0.5, -3) {$f$};
		\node [style=none] (21) at (-4.25, 2.25) {\tiny $1$};
		\node [style=none] (24) at (-4.25, -0.75) {\tiny $1$};
		\node [style=none] (25) at (-4.5, -3.25) {\tiny $1$};
		\node [style=none] (26) at (-2.75, -2.25) {\tiny $2$};
		\node [style=none] (27) at (-1.25, -0.75) {\tiny $1$};
		\node [style=node] (28) at (5, 3) {$a$};
		\node [style=node] (29) at (3.75, 0) {$b$};
		\node [style=node] (30) at (6.5, 0) {$c$};
		\node [style=node] (31) at (5, -3) {$e$};
		\node [style=node] (32) at (2, -3) {$d$};
		\node [style=node] (33) at (8, -3) {$f$};
		\node [style=none] (34) at (4.25, 2.25) {\tiny $1$};
		\node [style=none] (35) at (4.25, -0.75) {\tiny $1$};
		\node [style=none] (36) at (4, -3.25) {\tiny $1$};
		\node [style=none] (37) at (5.75, -2.25) {\tiny $2$};
		\node [style=none] (38) at (7.25, -0.75) {\tiny $1$};
		\node [style=none] (40) at (5.25, 2) {\tiny $2$};
		\node [style=none] (41) at (4, 3) {\tiny $3$};
		\node [style=none] (42) at (5.75, 2.25) {\tiny $4$};
		\node [style=none] (43) at (6, 3) {\tiny $5$};
		\node [style=none] (44) at (3, -0.75) {\tiny $2$};
		\node [style=none] (45) at (6, -3.25) {\tiny $3$};
		\node [style=none] (47) at (13, 0) {\Large $[a,b,e,d,c,f]$};
		\node [style=none] (48) at (4.75, -0.5) {\tiny $4$};
		\node [style=none] (49) at (4.75, 0.25) {\tiny $3$};
		\node [style=none] (50) at (0.75, 0) {\large $\Rightarrow$};
		\node [style=none] (51) at (9.5, 0) {\large $\Rightarrow$};
	\end{pgfonlayer}
	\begin{pgfonlayer}{edgelayer}
		\draw [style=edge] (0) to (1);
		\draw [style=edge] (0) to (2);
		\draw [style=edge] (2) to (5);
		\draw [style=edge] (1) to (3);
		\draw [style=edge] (1) to (4);
		\draw [style=edge] (3) to (4);
		\draw [style=edge] (5) to (3);
		\draw [style=edge] (3) to (2);
		\draw [style=edge] (15) to (16);
		\draw [style=edge] (17) to (20);
		\draw [style=edge] (16) to (18);
		\draw [style=edge] (18) to (19);
		\draw [style=edge] (18) to (17);
		\draw [style=edge] (30) to (33);
		\draw [style=edge] (29) to (31);
		\draw [style=edge] (31) to (32);
		\draw [style=edge] (31) to (30);
		\draw [style=edge] (28) to (31);
		\draw [style=edge] (29) to (32);
		\draw [style=edge] (31) to (33);
		\draw [style=edge, bend left=45, looseness=0.75] (28) to (33);
		\draw [style=edge] (28) to (30);
		\draw [style=edge, bend right=45, looseness=0.75] (28) to (32);
		\draw [style=edge] (28) to (29);
		\draw [style=edge] (29) to (30);
		\draw [style=edge] (29) to (33);
	\end{pgfonlayer}
\end{tikzpicture}}
    \caption{The construction of the lexicographic depth-first traversal starting from $a$ on a given edge-ordered graph. First we extract the least-path tree, transitively close it, then isolate the ordered neighborhood of $a$. Numerals indicate the edge ordering. Each transformation is universal, except for a ``silent'' (forgetful) transformation fixing the transitive graph but forgetting some structure on morphisms.}
    \label{fig:traversal}
\end{figure}

In \cite{Abr20}, Abramsky describes a ``great divide'' between those areas of theoretical computer science focused on \emph{structure} (semantics and type theory), and those focused on \emph{power} (computability and complexity theory); they have ``almost disjoint communities'' with ``no common technical language or tools.'' He proposes a high-level ``theory-building'' program of integrating these approaches with the intent of solving hard problems, akin to Grothendieck's program in algebraic geometry. We envision developing a theory of \emph{compositional graph algorithms through universal properties} as a step along this way.
A long-term goal of such a project would be to see whether one could recover \emph{algorithms} from \emph{problem statements}, if the latter are suitably formulated.

This paper is structured as follows: We describe the necessary background on graphs and (lexicographic) traversals in Section~\ref{sec:bg}, and present an alternate formulation of traversals in Section~\ref{sec:alternate} that the remaining work will build on. We go on to describe the categories of edge-ordered graphs on which our work is founded in Section~\ref{sec:graphs}, and
present the two categories of least path trees and universal constructions associated with lexicographic breadth-first and depth-first traversals respectively in Section~\ref{sec:least dfs trees} and Section~\ref{sec:least bfs trees}. The categories of transitively closed least path trees and their universal constructions are presented in Section~\ref{sec:traversals_as_transitive_lex_graphs} and Section~\ref{sec:short-lex transitive closure}. Finally, we summarise our categorical construction of lexicographic breadth-first and depth-first traversals by universal means in Section~\ref{sec:putting}, and end with some concluding remarks in Section~\ref{sec:discussion}.

Note that in the interests of space, we postpone all proofs until two appendices: one for the pre-categorical development (Sections \ref{sec:bg} and \ref{sec:alternate}), and another for category theory (Sections \ref{sec:graphs} through Section \ref{sec:putting}).

\section{Background and Preliminaries}\label{sec:bg}

Our objective is to give a categorical formulation of the canonical breadth-first and depth-first traversals of directed edge-ordered graphs. But first, in Sections \ref{sec:bg} and \ref{sec:alternate}, we give a ``pre-categorical'' account of these traversals in terms of extremal-path trees, which motivates the subsequent categorical treatment in Sections \ref{sec:graphs} through \ref{sec:short-lex transitive closure}.
The main results of Sections \ref{sec:bg} and \ref{sec:alternate} can be found in \cite{TK95}, but we provide our own presentation. 

\begin{definition}
  A (directed) graph is a pair $(V, {\to})$ where $V$ is the
  set of \emph{vertices} and $\to \subseteq V \times V$ is the \emph{edge
  relation}.
\end{definition}

\begin{definition}
  The \emph{neighborhood} of a vertex $v \in V$, denoted $N(v)$, is the set of
\emph{outgoing} edges of $v$; that is, $N(v) = {\to} \cap \{(v,x) \mid x \in
V\}$. A graph is \emph{pointed} when it has a distinguished vertex $v_0$.
\end{definition}

\begin{definition}
  A \emph{path} in a graph is a finite sequence of vertices $v_1 \to v_2 \to
\cdots \to v_n$; say that $v_1$ is the \emph{source} of the path, and $v_n$ the \emph{target}. We will write $u \rightsquigarrow v$ to say that there is a path
from $u$ to $v$, and $u \xrightsquigarrow{\pi} v$ about a specific such path
named $\pi$. We say that a path is \emph{proper} when no vertex occurs more
than once in the path, i.e., when $v_i = v_j$ implies $i=j$ for all $i,j \in I$. We use $\varepsilon$ to denote the unique \emph{empty} (proper) path in each neighborhood.

\end{definition}

\begin{definition}
  
Two paths are \emph{co-initial} in case they share a source, and \emph{co-final} in case they share a target. If the source of $\sigma$ agrees with the target of $\pi$, we can \emph{compose} them to obtain $\pi \sigma$.   For two co-initial paths $\pi$ and $\sigma$, we say $\pi \sqsubset \sigma$ (read: $\sigma$ \emph{extends} $\pi$) in case $\pi$ is an initial subsequence (i.e., a prefix) of $\sigma$. 
\end{definition}

\begin{remark}
Notice that $u \rightsquigarrow v$ iff there exists a proper path from $u$ to $v$, as repetitions in a path can always be deleted.
\end{remark}

\begin{definition}
 A pointed graph $(G,v_0)$ is \emph{connected} if for every vertex $v$, there exists a path $v_0 \rightsquigarrow v$.
\end{definition}

\subsection{Graph searching and traversals}
By a \emph{graph search}, we mean an algorithm for visiting all the vertices in a connected graph, starting at a given source. Two of the most important types of graph search are \emph{depth-first} and \emph{breadth-first} :

\begin{definition}[Depth-first search]\label{def:depth-first search}
Given as input a finite connected graph $G = (V,E)$, initialize a list $L=()$, and a stack $S=(v_0)$ for some vertex $v_0$. While $S$ is nonempty, pop the first element $v$ from $S$. If $v$ is already contained in $L$, go back to the start of the loop. Otherwise, let $L = (L,v)$, and push every vertex in $\partial v$ onto $S$, where $\partial v = N(v) \setminus L$.
\end{definition}

\begin{definition}[Breadth-first search]\label{def:breadth-first search}
 Given as input a finite connected graph $G = (V,E)$, initialize a list $L=()$, and a queue $Q=(v_0)$ for some vertex $v_0$. While $S$ is nonempty, dequeue the front element $v$ from $Q$. If $v$ is already contained in $L$, go back to the start of the loop. Otherwise, let $L = (L,v)$, and enqueue every vertex in $\partial v$ to the back of $Q$, where $\partial v = N(v) \setminus L$.
\end{definition}
Note that depth-first and breadth-first search are nondeterministic, in the sense that we do not specify which order to add vertices from $\partial v$ in. Moreover, vertices may occur more than once in $S$ or $Q$, as the same vertex may be added as a neighbor multiple times. However, vertices may not occur more than once in $L$; when depth-first or breadth-first search is complete, $L$ lists all the vertices in $G$.

As we vary over all the nondeterministic traces of depth-first search or breadth-first search over a graph $G$, the resulting orderings are depth-first or breadth-first traversals.

\begin{definition}\label{def:DFT and BFT}
 A linear ordering $<$ of the vertices of a graph $G$ is a \emph{depth-first traversal}, respectively \emph{breadth-first traversal} of $G$ if there exists some computation of depth-first search, respectively breadth-first search, on $G$ according to which vertices are added to $L$ in exactly the order $<$.
\end{definition}
Both depth- and breadth-first traversals have important characterizations in the first-order language of ordered graphs; these are due to Corneil and Krueger \cite{CK08}, who state them in the special case of undirected graphs. Here, we only need to know that they are necessary.

\begin{lemma}\label{lem:FO characterization of DFT}
If $\cdot < \cdot$ is a depth-first traversal of a finite, connected graph $G$, and if there are vertices $u < v < w$ such that $u \to w$, then there exists some $v'$ such that $u \le v' < v$ and $v' \to v$.
\end{lemma}

\begin{lemma}\label{lem:FO characterization of BFT}
If $\cdot < \cdot$ is a breadth-first traversal of a finite, connected graph $G$, and if there are vertices $u < v < w$ such that $u \to w$, then there exists some $v'$ such that $v' \le u < v$ and $v' \to v$.
\end{lemma}
These characterizations motivate the following definitions.

\newcommand{\dfp}{\mathit{dfp}}

\begin{definition}
 Let $G$ be a finite connected graph and $\cdot < \cdot$ be a depth-first traversal of $G$. Then for any vertex $v$, define the \emph{depth-first predecessor} $\dfp(v)$ to be the greatest $u < v$ such that $u \to v$, if $v$ is not the minimal vertex. For the minimal vertex $v$, let $\dfp(v) = v$.
\end{definition}

\newcommand{\bfp}{\mathit{bfp}}

\begin{definition}
 Let $G$ be a finite connected graph and $\cdot < \cdot$ be a depth-first traversal of $G$. Then for any vertex $v$, define the \emph{breadth-first predecessor} $\bfp(v)$ to be the least $u < v$ such that $u \to v$, if $v$ is not the minimal vertex. For the minimal vertex $v$, let $\bfp(v) = v$.
\end{definition}
Depth- and breadth-first predecessors allow for  elegant restatements of Lemmata \ref{lem:FO characterization of DFT} and \ref{lem:FO characterization of BFT}: if $\cdot < \cdot$ is a depth-first traversal of $G$, then for any $v < w$, if $\dfp(w) < v$, then $\dfp(w) \le \dfp(v)$. Similarly, if $\cdot < \cdot$ is a breadth-first traversal of $G$ then for any $v < w$, if $\bfp(w) < v$, then $\bfp(v) \le \bfp(w)$.
In fact the second condition is equivalent to saying that $\bfp$ is weakly monotone, viz., $v \le w \implies \bfp(v) \le \bfp(w)$.

For the next definition, notice that the orbits $\{v,\dfp(v),\dfp^2(v),\dots\}$ and 
$\{v,\bfp(v),\bfp^2(v),\dots\}$ of any vertex $v$ are finite, and their 
\emph{reversals} are paths from the least vertex $v_0$ to $v$.

\begin{definition}\label{def:df and bf canonical paths}
 Let $G$ be a finite connected graph, $\cdot < \cdot$ be a depth-first traversal of $G$, and $v_0$ be the $<$-least vertex. Then the \emph{canonical df-path} from $v_0$ to any vertex $v$ is the path $v_0 \to v_1\to \dots \to v_{\ell - 1}$ such that $v_{\ell -1} = v$, and for every $0 \le i < \ell-1$, $\dfp(v_{i+1}) = v_i$.
 
 Analogously, the \emph{canonical bf-path} from $v_0$ to any vertex $v$ is the path $v_0 \to v_1\to \dots \to v_{\ell - 1}$ such that $v_{\ell -1} = v$, and for every $0 \le i < \ell-1$, $\bfp(v_{i+1}) = v_i$.
\end{definition}

\subsection{Lexicographic searching}\label{subsec:lexdfs}

The objective of this paper is to show that the depth-first and breadth-first traversals of a graph is \emph{canonical} in some precise categorical sense. Of course, this cannot be true at face value: in a complete graph, every ordering of the vertices is both a breadth-first and a depth-first traversal, and none of them is canonical.

As we remarked in the introduction, an edge-ordering is precisely the amount of additional structure we need. Over such graphs, we can make searching deterministic, as we now show.

\begin{definition}\label{def:EO graphs}
  A \emph{finite edge ordered graph} is a finite graph where each neighborhood is equipped with a (strict) linear order $\cdot \edgel \cdot$.
\end{definition}

\begin{definition}[Lexicographic depth-first search]
Let $G$ be a finite, connected, pointed, edge-ordered graph with distinguished vertex $v_0$. Initialize a list $L=()$ and a stack $S=(v_0)$. While $S$ is nonempty, pop the first element $v$ from $S$. If $v$ is already contained in $L$, go back to the start of the loop. Otherwise, let $L = (L,v)$, and push every vertex in $\partial v$ onto $S$ in \emph{reverse} $\edgel$-order, where $\partial v$ is the set of neighbors of $v$ not in $L$.
\end{definition}

\begin{remark}
We push vertices from $\partial v$ onto $S$ in reverse $\edgel$-order, so that the least vertices from $\partial v$ end up on top of $S$.
\end{remark}

\begin{definition}[Lexicographic breadth-first search]
Let $G$ be a finite, connected, pointed, edge-ordered graph with distinguished vertex $v_0$. Initialize a list $L=()$ and a queue $Q=(v_0)$. While $Q$ is nonempty, dequeue first element $v$ from $S$. If $v$ is already contained in $L$, go back to the start of the loop. Otherwise, let $L = (L,v)$, and enqueue every vertex from $\partial v$ onto $Q$ in $\edgel$-order, where $\partial v$ is the set of neighbors of $v$ not in $L$.
\end{definition}

\begin{definition}\label{def:lex bf and df traversals}
 The depth-first traversal computed by lexicographic depth-first search on a finite, connected, pointed, edge-ordered graph $G$, is its \emph{lexicographic depth-first traversal}. Similarly, the breadth-first traversal computed by lexicographic breadth-first search is its \emph{lexicographic breadth-first traversal}.
\end{definition}

\begin{remark}
The usage of \emph{lexicographic (depth, breadth)-first (search, traversal)} is ambiguous in the literature. Here, we use it in the same way as Delatorre and Kruskal \cite{TK95,TK01}; however, the \emph{lexicographic breadth-first search} of Rose and Tarjan \cite{RT75} and the analogous \emph{depth-first} version of Corneil and Krueger \cite{CK08} are different. The latter are further refinements of breadth-first search and depth-first search over graphs, not a ``determinization'' by an edge-ordering.
\end{remark}

\section{An alternate formulation}\label{sec:alternate}

In this section, we give different characterizations of the lexicographic depth-first and breadth-first traversals, independent of any graph search, which suggest the category-theoretic treatment that occupies us for the rest of the paper.

\begin{definition}\label{def:path_order}
   Fix an edge-ordered graph $G$. Define the \emph{lexicographic path relation} $\cdot \prec \cdot$ on any proper co-initial paths $\pi$ and $\sigma$ in $G$ as follows:
  \begin{enumerate}[(i)]
    \item\label{def:extends sqsubet} if $\pi \sqsubset \sigma$, then $\pi \prec \sigma$, similarly if $\sigma \sqsubset \pi$, then $\sigma \prec \pi$; otherwise,
    \item\label{def:compares first different edge} let $\zeta$ be the longest common prefix of $\pi$ and $\sigma$, let $u$ be the target of $\zeta$, and let $v_1$ and $v_2$ be the first vertices immediately following $\zeta$ in $\pi$ and $\sigma$ respectively. Order $\pi \prec \sigma \iff u \to v_1 \edgel u \to v_2$.
  \end{enumerate}
\end{definition}

\begin{remark}\label{rem:path_order props}
 The following properties hold of $\prec$:
 \begin{enumerate}
     \item\label{rem:empty path least} The empty path $(v)$ is least among all paths from $v$.
     
     \item\label{rem:right mult} If $\pi \prec \sigma$ and $\pi \not \sqsubseteq \sigma$, then for any $\alpha$ and $\beta$, $\pi \alpha \prec \sigma \beta$.
     
     \item\label{rem:left mult} If $\pi \prec \sigma$, then $\alpha \pi \prec \alpha \sigma$.
     
     \item\label{rem:left cancellation} If $\alpha \pi \prec \alpha \sigma$, then $\pi \prec \sigma$.
 \end{enumerate}
\end{remark}
Since the set of proper paths is finite, any subset has a $\prec$-least element, which justifies the next definitions.

\begin{definition}\label{def:lex min path}
  For vertices $u,v$ in a finite edge-ordered graph, let $
  \min(u \tp v)$ be the lexicographically least proper path from $u$ to $v$.
\end{definition}

\begin{definition}\label{def:slex min path}
  For vertices $u,v$ in a finite edge-ordered graph, let $
  \min^s(u \tp v)$ be the lexicographically least \emph{shortest} path from $u$ to $v$.
\end{definition}
In fact, it will be convenient to define the following relation, the \emph{shortlex} order:
\begin{definition}
 Let $\pi \prec^s \sigma$ mean that either $|\pi| < |\sigma|$, or $|\pi| = |\sigma|$ and $\pi \prec \sigma$.
\end{definition}
In this case, $\min^s(u \tp v)$ is simply, the $\prec^s$-least path $u \tp v$.

In the remainder of this section, we fix a finite, connected, pointed, edge-ordered graph $G$ with distinguished element $v_0$. Reserve the symbol $\cdot \edgel \cdot$ for the given ordering on co-initial edges and $\cdot \prec \cdot$ for the induced lexicographic ordering on co-initial paths. Let $\cdot \lexdft \cdot$ and $\cdot \lexbft \cdot$ denote the lexicographic depth-first and breadth-first traversals respectively. Let $\mathfrak{P}v$ and $\mathfrak{Q}v$ denote the canonical df- and bf-paths from $v_0$ to $v$ with respect to $\cdot \lexdft \cdot$ and $\cdot \lexbft \cdot$ respectively.

Our goal is to prove the following ``alternate characterizations'' of the lexicographic depth- and breadth-first traversals:
$$ u \lexdft v \iff \mathfrak{P}u \prec \mathfrak{P}v $$
$$ u \lexbft v \iff \mathfrak{Q}u \prec^s \mathfrak{Q}v.$$
In the breadth-first case,
\newcommand{\lexbfte}{\le_B}

\begin{lemma}\label{lem:bf canonical paths are minimal}
If $u \lexbfte v$, then for any path $\pi : v_0 \tp v$, $|\mathfrak{Q}u| \le |\pi|$. In addition, if $|\mathfrak{Q}u| = |\pi|$, then $\mathfrak{Q}u \preceq \pi$.
\end{lemma}

\begin{corollary}\label{cor:bf canonical paths are minimal}
The following are all true of the operator $\mathfrak{Q}$:
\begin{enumerate}
    \item For any vertex $v$ of $G$, $\mathfrak{Q}v = \min^s(v_0 \tp v)$.
    \item $u \lexbfte v$ iff $\mathfrak{Q}u \preceq^s \mathfrak{Q}v$.
\end{enumerate}
\end{corollary}
Corollary \ref{cor:bf canonical paths are minimal} gives us our desired alternate characterization of the lexicographic breadth-first traversal, viz., $u \lexbft v \iff \mathfrak{Q}u \prec^s \mathfrak{Q}v$. In the depth-first case,

\begin{definition}\label{def:prec D}
Define $v \prec_D w$ in case $\min(v_0 \tp v) \prec \min(v_0 \tp w)$.
\end{definition} 

\begin{lemma}\label{lem:canonical paths are minimal}
For any vertex $v$ of $G$, $\mathfrak{P}v = \min(v_0 \tp v)$.
\end{lemma}

\begin{theorem}\label{thm:alternate-formulation}
  The order $\cdot \prec_D \cdot$ is exactly the lexicographic depth-first traversal $\cdot \lexdft \cdot$ of $G$.
\end{theorem}

\section{Categories of graphs} 
\label{sec:graphs}

We now work towards a categorical formulation of the lexicographic depth-first traversal based on Theorem \ref{thm:alternate-formulation}. This means we need to categorify, i.e., equip with morphisms, all the objects we defined in Section \ref{sec:bg}, as well as introduce some new objects.

Continuing the convention established above, we use they symbol $\prec$ and variations thereof to refer to orderings of co-initial paths, and the symbol $\edgel$ to refer to neighborhood orders, i.e., orderings of co-initial edges. We reserve the symbol $<$ for vertex orders, but these will not reappear until Section \ref{sec:putting}.

\begin{definition}
  A \emph{homomorphism of graphs} $G \tot{h} H$ is a function from 
  $G$-vertices to $H$-vertices which preserves edge connectivity: For all $G$-vertices $u,v$, $u \to v$ in $G$ implies $h(u) \to h(v)$ in $H$.
  A \emph{homomorphism of pointed graphs} $G \tot{h} H$ must additionally map the distinguished vertex of $G$, and \emph{only} that vertex, to the distinguished vertex of $H$.
\end{definition}

\begin{definition}
  If $h : G \to H$ is a graph homomorphism and $\pi$ is the path $v_1 \to v_2 \to \dots v_n$ in $G$, then $h(\pi)$ is defined to be the path $h(v_1) \to h(v_2) \to \dots h(v_n)$ in $H$.
\end{definition}

\begin{remark}\label{rem:basic path properties}
The following hold of any homomorphism $h : G \to H$;
\begin{enumerate}
    \item $h(\varepsilon) = \varepsilon$, $h(\pi \sigma) = h(\pi) h(\sigma)$, and if $\pi \sqsubset \sigma$, then $h(\pi) \sqsubset h(\sigma)$.
    
    \item If $h$ is injective on vertices and $\pi$ is a proper path, then $h(\pi)$ is a proper path.
    
    \item If $h$ is injective on vertices, then $h$ preserves longest common prefixes.
\end{enumerate}
\end{remark}
We now define homomorphisms of edge-ordered graphs (cf. Definition \ref{def:EO graphs}). In addition to the ``straightforward'' property of being monotone on neighborhoods, we also want to consider homomorphisms that preserve lexicographically minimal paths.
This gives us two refinements of the notion of homomorphism, depending on whether we want to preserve lexicographically ($\prec$) least paths or short-lex ($\prec^s$) least paths.

\begin{definition}\label{def:EO_hom}
 A \emph{homomorphism} of finite, pointed, edge-ordered graphs $G \tot{h} H$ is a homomorphism of pointed graphs that is monotone on neighborhoods; explicitly,
   \begin{enumerate}[(i)]
    \item\label{def:EO_hom:edge} $u \to v$ in $G$ implies $h(u) \to h(v)$ in 
    $H$,
    \item\label{def:EO_hom:point} $v_G$ is the \emph{unique} vertex such that $h(v_G) = 
    v_H$, where $v_G$ and $v_H$ are the distinguished points of $G$ and $H$ respectively, and
    \item\label{def:EO_hom:monotone} $u \to v_1 \edgel u \to v_2$ implies $h(u) \to
    h(v_1) \edgel h(u) \to h(v_2)$.
  \end{enumerate}
  In addition, a \emph{lex-homomorphism} must preserve minimal paths:
  \begin{enumerate}[(i)]
    \setcounter{enumi}{3}
    \item\label{def:EO_hom:min} $h(\min(u \tp v)) = \min(h(u) \tp h(v))$,
  \end{enumerate}
  and a \emph{short-lex homomorphism} must preserve minimal shortest paths:
  \begin{enumerate}[(i)]
      \setcounter{enumi}{4}
      \item\label{def:EO_hom:shortmin} $h(\min^s(u \tp v)) = \min^s(h(u) \tp h(v))$.
  \end{enumerate}
\end{definition}

\begin{lemma}\label{lem:EO homs preserve lex order}
If $h : G \to H$ is a homomorphism of edge-ordered graphs then $h$ preserves $\cdot \prec \cdot$ as well as $\cdot \prec^s \cdot$.
\end{lemma}
An important special case of edge-ordered graphs are those where the edge order agrees with the path order in each neighborhood.

\begin{definition}\label{def:lex_graph}
  A \emph{lex-graph} is a pointed, connected, edge-ordered finite graph such that the edge order is compatible with the lexicographic path order;
  explicitly, it is a finite directed graph equipped with
  \begin{enumerate}[(i)]
    \item\label{def:lex_graph:point} a distinguished vertex $v_0$, and
    \item\label{def:lex_graph:order} a linear order $\edgel$ on each neighbourhood $N(u)$ such that for every $v_1, v_2 \in N(u)$, $\min(u \tp v_1) \prec \min(u \tp v_2) \iff u \to v_1 \edgel u \to v_2$.
  \end{enumerate}
\end{definition}

\begin{remark}
We might be tempted to define, analogously, a \emph{short-lex graph} by demanding that  $v_1, v_2 \in N(u)$, $\min^s(u \tp v_1) \prec^s \min^s(u \tp v_2) \iff u \to v_1 \edgel u \to v_2$. However notice that this condition simply holds automatically for any edge-ordered graph: if there is an edge $u \to v$, that edge is the lexicographically least shortest path.
\end{remark}
Finally, we isolate two extremal special cases of lex-graphs, one with very few edges, and one with very many.

\begin{definition}
 An \emph{arborescence} is a pointed directed graph $G = (V,\to,v_0)$ such that for every vertex $u$, there is a unique path $v_0 \tp u$ in $G$.
\end{definition}

\begin{remark}
If $(V,\to,v_0)$ is an arborescence, its underlying undirected graph $(V,E)$ is connected and acyclic (where $E(u,v)$ iff $(u \to v)$ or $(v \to u)$), and every edge $u \to v$ is oriented away from $v_0$, meaning that the distance from $v_0$ to $u$ is less than the distance from $v_0$ to $v$. 
If $G$ is a finite edge-ordered graph that is also an arborescence, then it is already lex-graph, since the only path between $u$ and $v \in N(u)$ \emph{is} the edge $u \to v$.
Moreover, if $S$ and $T$ are arborescences and if $h : S \to T$ is a pointed, edge-ordered graph homomorphism, then it is already both a lex-homomorphism and a short-lex homomorphism.
\end{remark}

\begin{definition}
 A graph is \emph{transitive} if its edge relation is; i.e., if $u \to v$ and $v \to w$ implies $u \to w$. A \emph{transitive lex-graph} is simply a lex-graph with a transitive edge relation.
\end{definition}

It follows readily that this zoo of graph variations each form a category,
\begin{itemize}
    \item $\mathbf{FinGraph}_<^l$ of connected, finite pointed edge ordered graphs with homomorphisms,
    \item $\mathbf{FinGraph}_\star^l$ of connected, finite pointed edge ordered graphs with lex-homomorphisms,
    \item $\mathbf{FinGraph}_\star^s$ of connected, finite pointed edge ordered graphs with short-lex homomorphisms,
    \item $\mathbf{LexGraph}$ of lex-graphs with lex-homomorphisms, 
    \item $\mathbf{FinArb}_\star^<$ of finite, edge-ordered arborescences with edge-ordered graph homomorphisms, 
    \item $\mathbf{TLexGraph}$ of transitive lex-graphs with lex-homomorphisms,
\end{itemize}
Diagrammatically, we summarize the relationships between these categories in Figure~\ref{fig:category zoo}, where each arrow indicates inclusion as a full subcategory.

\begin{figure}
    \centering
\[\begin{tikzcd}
	{\mathbf{TLexGraph}} & {\mathbf{LexGraph}} & {\mathbf{FinGraph}^l_\star} & {\mathbf{FinGraph}^<_\star} \\
	& {\mathbf{FinArb}^<_\star} & {\mathbf{FinGraph}^s_\star}
	\arrow[from=1-2, to=1-3]
	\arrow[from=2-2, to=1-2]
	\arrow[from=1-1, to=1-2]
	\arrow[from=2-2, to=2-3]
	\arrow[from=1-3, to=1-4]
	\arrow[from=2-3, to=1-4]
\end{tikzcd}\]
    \caption{Categories of edge-ordered graphs and their relationships.}
    \label{fig:category zoo}
\end{figure}
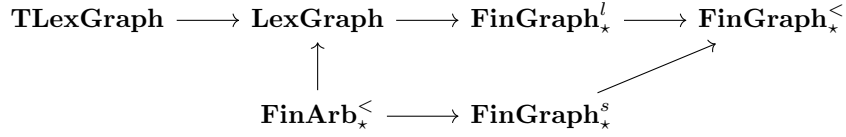

There is one additional category that will be introduced in Section \ref{sec:short-lex transitive closure}, and that is $\mathbf{TArb}$ (Definition \ref{def:cat TArb}) whose objects are transitive closures of arborescences with a particular edge order depending on the underlying arborescence. Curiously, their morphisms preserve \emph{longest} paths instead of shortest paths.


\newcommand{\T}{\Theta}
\newcommand{\I}{I}

\section{Least path trees}\label{sec:least dfs trees}
Given a finite, connected, pointed, edge-ordered graph, we can delete all edges which are not on some lexicographically least path starting from the distinguished vertex $v_0$. We get a finite, edge-ordered arborescence. In this section, we show that this construction (denoted $\T$) is cofree, indeed coreflective: it is right adjoint to the inclusion functor $\I$.

\begin{definition}[The functor $\T$]\label{def:functor T}
 Given a finite, connected, pointed, edge-ordered graph $G$, with distinguished vertex $v_0$, the graph $\T(G)$ is defined as follows:
 \begin{itemize}
     \item the vertices of $\T(G)$ are the vertices of $G$, and
     \item any edge $u \to v$ appears in $\T(G)$ iff $u \to v$ is contained in $\min(v_0 \tp v)$.
     \item Order co-initial edges $u \to v_1 \edgel u \to v_2$ in $\T(G)$ iff the same relation holds in $G$.
 \end{itemize}
 For a lex-homomorphism $h : G \to H$ define $\T(h):\T(G) \to \T(H)$ by $\T(h)(v) = h(v)$ for any vertex $v \in \T(G)$.
\end{definition}
The proof that $\T$ is a well-defined functor into the indicated category is in the appendix (Lemma \ref{lem:theta well defined}). Next we define $\I$, which is simply inclusion of categories.

\begin{definition}[The functor $\I$]\label{def:functor I}
 Given a finite, edge-ordered arborescence $T$, the object $\I(T)$ is simply identified with $T$. Given a pointed, edge-ordered homomorphism $h : S \to T$, the homomorphism $\I(h):\I(S) \to \I(T)$ is simply identified with $h$.
\end{definition}
Since every finite, edge-ordered arborescence is a finite, connected, pointed edge-ordered graph, and since morphisms in $\mathbf{FinArb}^<_\star$ are defined to be morphisms of pointed, edge-ordered graphs, $\I$ is well-defined.

\begin{theorem}\label{thm:adj graphs with arbs (lex)}
  There is an adjunction $I \dashv \Theta$.
\end{theorem}

Since the image of $\I$ is always a lex-graph, and since $\mathbf{LexGraph}$ is a full subcategory of $\mathbf{FinGraph}^l_\star$, we actually get another adjunction for free, namely $I' \dashv \Theta'$
where $\I'$ is the functor $\I$ but with target category \textbf{LexGraph}, and $\T'$ is $\T$ restricted to \textbf{LexGraph}. The proof is in the appendix (Lemma \ref{lem:adjunction for free}).

\section{Least shortest-path trees}\label{sec:least bfs trees}
Our objective in this section is to establish a shortest-paths analogue of the previous section, viz., $I \dashv S$
where $\mathbf{FinGraph}_\star^s \tot{S} \mathbf{FinArb}^<_\star$ remembers only the edges on lexicographically least \emph{shortest} paths and $\mathbf{FinArb}^<_\star \tot{\I} \mathbf{FinGraph}_\star^s$ is the inclusion functor. (It is not, of course, the same $\I$ as in the previous section, the source category being different.)

\begin{definition}[The functor $S$]\label{def:functor S}
 Given a finite, connected, pointed, edge-ordered graph $G$, with distinguished vertex $v_0$, the graph $S(G)$ is defined as follows:
 \begin{itemize}
     \item the vertices of $S(G)$ are identified with the vertices of $G$, and
     \item any edge $u \to v$ appears in $S(G)$ iff $u \to v$ is contained in $\min^s(v_0 \tp v)$.
     \item Order co-initial edges $u \to v_1 \edgel u \to v_2$ in $S(G)$ iff the same relation holds in $G$.
 \end{itemize}
 If $h : G \to H$ is a homomorphism, define $S(h):S(G) \to S(H)$ by $S(h)(v) = h(v)$, for any vertex $v \in S(G)$.
\end{definition}
To show that $S$ is functorial, it suffices to check that $S(G)$ is always an arborescence. We prove this in the appendix (Lemma \ref{lem:S well defined}).

\begin{theorem}\label{thm:adj graphs with arbs (short-lex)}
  There is an adjunction $\I \dashv S$.
\end{theorem}

\section{Transitive closure of lex-graphs} 
\label{sec:traversals_as_transitive_lex_graphs}
There is a well-known adjunction between the category of graphs and the category
$\mathbf{TGraph}$ of transitive graphs, where the functor in one direction transitively closes the edge relation, and in the other direction is simply inclusion \cite{Mas20}. Here, we refine this to an adjunction between
$\mathbf{LexGraph}$ and $\mathbf{TLexGraph}$. Indeed, the usual transitive closure appears when forgetting the
order

While it is immediately clear that there is a forgetful (inclusion) functor
$\mathbf{TLexGraph} \tot{\U} \mathbf{LexGraph}$, we must construct the free functor in the other direction.

\begin{definition}[The functor $\F$]
 Given a lex-graph $G$, define the transitive lex-graph $\F(G)$ as follows:
 \begin{itemize}
     \item the vertices of $\F(G)$ are the vertices of $G$,
     \item the distinguished point of $\F(G)$ is the distinguished point of $G$,
     \item The edge relation of $\F(G)$ is the transitive closure of the edge relation of $G$, and finally
     \item for any vertices $v_1,v_2$ in a common neighborhood $N(u)$, let $u \to v_1 \edgel u \to v_2$ in $\F(G)$ iff $\min(u \tp v_1) \prec \min(u \tp v_2)$ in $G$.
 \end{itemize}
 On morphisms, given $G \tot{h} H$ we define $\F(G) \tot{\F(h)} \F(H)$ by
  $\F(h)(v) = h(v)$.
\end{definition}
We must show that $F$ is a well-defined functor, but before doing so, we state several important relationships between the lex-graph $G$ and the edge-ordered graph $\F(G)$:

\begin{lemma}\label{lem:min}
For any lex-graph $G$,
\begin{enumerate}[(i)]
    \item The neighborhood order of $\F(G)$ extends that of $G$
    
    \item The lexicographic path order of $\F(G)$ extends that of $G$.
    
    \item Minimal paths in $G$ and $\F(G)$ coincide.
\end{enumerate}

\end{lemma}

\begin{lemma}\label{lem:free_functor}
  The map $\mathbf{LexGraph} \tot{\F} \mathbf{TLexGraph}$ is well-defined and functorial.
\end{lemma}

\begin{theorem}\label{thm:lex/tlex adjunction}
  There is an adjunction $F \dashv U$.
\end{theorem}

\begin{remark}
In the title of this paper, we promise universal constructions -- and while such do arise from adjunctions (from the unit and counit), it seems fitting that we make this explicit at least once. The reader is invited to extract similar universal mapping properties from other adjunctions in this paper.

As a consequence of Theorem~\ref{thm:lex/tlex adjunction}, for any lex-graph $G$ in with lexicographic-transitive closure $U(F(G))$, given any other transitive lex-graph $U(G')$ and $G \tot{h} U(G')$ there is a \emph{unique} homomorphism of transitive lex-graphs $F(G) \tot{\hat{h}} G'$ such that $h$ factors as $G \hookrightarrow U(F(G)) \tot{U(\hat{h})} U(G')$.
\end{remark}

\section{Transitive closure of least shortest path trees} 
\label{sec:short-lex transitive closure}
We now establish the final adjunction of our paper, and the second adjunction needed to characterize breadth-first traversals. Analogously to the depth-first case, this adjunction relates $\mathbf{FinArb}^<_\star$ and a category of transitively closed graphs:

\begin{definition}[The category $\mathbf{TArb}$]\label{def:cat TArb}
 A finite, pointed, connected, edge-ordered graph $G$ is an object of $\mathbf{TArb}$ iff there exists a finite, connected, pointed, edge-ordered arborescence $T$ and an identification of the vertices of $G$ with the vertices of $T$ such that
 \begin{itemize}
     \item the distinguished points of each are identical,
     \item the edge relation of $G$ is the transitive closure of the edge relation of $T$, and
     \item $u \to v_1 \edgel u \to v_2$ in $G$ iff $v_0 \tp v_1 \prec^s v_0 \tp v_2$ in $T$, where $v_0 \tp v$ is the \emph{unique} path from $v_0$ to $v$ in $T$.
 \end{itemize}
 A map $G_1 \tot{h} G_2$ is a morphism in $\mathbf{TArb}$ if it is a homomorphism of edge-ordered graphs ((\ref{def:EO_hom:edge})-(\ref{def:EO_hom:monotone}) of Defintion \ref{def:EO_hom}) and in addition it preserves \emph{longest} paths.
\end{definition}
This last property only makes sense if longest paths are unique. Luckily, longest paths in the transitive closure of an arborescence are exactly the edges of the original tree.

\begin{lemma}\label{lem:longest paths are tree edges}
 If $T$ is an arborescence (not necessarily edge-ordered) with root $v_0$, and $G$ is its transitive closure, then an edge $u \to v$ of $G$ is an edge of $T$ iff $u \to v$ appears on a longest path $v_0 \tp v$ in $G$. Moreover, the longest path $u \tp v$ in $G$ is unique.
\end{lemma}

\begin{remark}
As a consequence of Lemma \ref{lem:longest paths are tree edges}, for any graph $G \in \mathbf{TArb}$, $u \to v_1 \edgel u \to v_2$ in $G$ iff $v_0 \tp v_1 \prec^s v_0 \tp v_2$ in $G$, where $v_0 \tp v$ is the unique longest path from $v_0$ to $v$ in $G$.
\end{remark}
While the presentation of depth- and breadth-first traversals has thus far been similar even at a rather small scale, this adjunction diverges from the previous one in several ways: it does not decompose into an inclusion and a functor which lifts the ordinary transitive closure, and the curious preservation longest paths by morphisms is not suggested by the pre-categorical treatment of breadth-first traversals. This category arises somewhat mysteriously, and we have no justification for introducing it other than it works. 

We now define the two functors of the adjunction. The proofs that they are 
functorial are Lemmas \ref{lem:Gamma well defined} and \ref{lem:L is well-defined} in the appendix.

\newcommand{\G}{\Gamma}

\begin{definition}
 Define the functor $\mathbf{FinArb}^<_\star \tot{\G} \mathbf{TArb}$ by $\G(V,v_0,\to) = (V,v_0,\tot{\text{trans}})$, and for any vertices $(u,v_1,v_2)$ such that $u \to v_1$ and $u \to v_2 $ in $\G(T)$, define $u \to v_1 \edgel u \to v_2$ iff $v_0 \tp v_1 \prec^s v_0 \tp v_2$ in $T$. For morphisms $T \tot{h} T'$, define $\G(h):\G(T) \to \G(T')$ by $\G(h)(v) = h(v)$.
\end{definition}

\newcommand{\Ell}{\mathrm{L}} 
 
\begin{definition}
 Define the functor $\mathbf{TArb} \tot{\Ell} \mathbf{FinArb}^<_\star$ by identifying the set of vertices of $\Ell(G)$ with the vertices of $G$, and including an edge $u \to v$ in $\Ell(G)$ iff it lies on the longest path $v_0 \tp v$ in $G$. The neighborhood order in $\Ell(G)$ is inherited from $G$, and for morphisms $G \tot{h} H$, define $\Ell(h) : \Ell(G) \to \Ell(H)$ by $\Ell(h)(v) = h(v)$.
\end{definition}

\begin{theorem}\label{thm:arb/TArb adjunction}
  There is an adjunction $\G \dashv \Ell$.
\end{theorem}


\section{Putting it all together}\label{sec:putting}

Suppose we fix a finite, pointed, edge-ordered graph, locate it in either $\mathbf{FinGraph}^l_\star$ or $\mathbf{FinGraph}^s_\star$, then apply the appropriate least-path tree and the transitive closure functors. Then the edge ordering of the distinguished vertex in the result is the lexicographic depth-first or breadth-first traversal respectively.

\begin{lemma}\label{lem:correctness of construction of DFT}
Given any $G \in \mathbf{FinGraph}^l_\star$, let $T = (\F \circ \I' \circ \T)(G)$. Then the neighborhood ordering $\edgel$ on the distinguished point in $T$ is the lexicographic depth-first traversal of $G$.
\end{lemma}

\begin{lemma}\label{lem:correctness of construction of BFT}
Given any $G \in \mathbf{FinGraph}^s_\star$, let $T = (\G \circ S)(G)$. Then the neighborhood ordering $\edgel$ on the distinguished point in $T$ is the lexicographic breadth-first traversal of $G$.
\end{lemma}
Finally, we show that these traversals can be extracted as vertex orders rather than edge orders, by defining a functor that takes an edge-ordered graph into the ordered neighborhood of its distinguished point. Let $\mathbf{FinLoset}$ denote the category of finite linearly ordered sets with monotone functions.

\begin{lemma}\label{lem:fingraph to finloset} 
There is a functor $\E : \mathbf{FinGraph}^<_\star \to \mathbf{FinLoset}$ that linearly orders the vertices in a given graph according to the ordering on $N(v_0)$.
\end{lemma}
Since there are inclusions $\mathbf{TLexGraph} \to \mathbf{FinGraph}^<_\star$ and $\mathbf{TArb} \to \mathbf{FinGraph}^<_\star$, we get

\begin{corollary}
There are functors $\mathbf{FinGraph}^l_\star \to \mathbf{FinLoset}$ and $\mathbf{FinGraph}^s_\star \to \mathbf{FinLoset}$ which compute the lexicographic depth-first and breadth-first traversals respectively.
\end{corollary}
Note that parts of this final step $\mathbf{FinGraph}^<_\star \to \mathbf{FinLoset}$ can also be made universal, but since we still have to accept the presence of inclusion functors without adjoints, we only sketch this construction: The functor $E$ from Lemma~\ref{lem:fingraph to finloset} restricts to a functor from the category of \emph{transitive} finite, pointed, edge-ordered graphs and their homomorphisms, and into the category of finite \emph{nonempty} linearly ordered sets with monotone functions \emph{which additionally preserve the least element}. This restricted functor has a left adjoint given by mapping a finite linearly ordered set $(V,<)$ to the graph with vertices $V$, distinguished vertex the least element $v_0$ of $V$ (guaranteed to exist when $V$ is finite), and edge relation given by $v_0 \to v$ for all $v \in V$ with $v \neq v_0$. This specializes to neither $\mathbf{TLexGraph}$ nor $\mathbf{TArb}$, however, since the counit of this adjunction need not preserve either least or longest paths.

\section{Discussion and open questions}\label{sec:discussion}
We have described a construction of depth-first traversals from a category of finite edge-ordered graphs whose morphisms preserve least paths as a series of universal constructions, and in analogous fashion, a construction of breadth-first traversals from a category of finite edge-ordered graphs whose morphisms preserve least shortest paths. 

This begs the question of whether there is way of generalizing these two constructions using a general notion of extremal path. Such a method is suggested by Delatorre and Kruskal \cite{TK95} using the machinery of a closed semiring system, which provides a general setting for several algebraic path problems, not only lexicographic breadth-first and depth-first search, but \emph{lexicographic topological search} as well. 

The deepest question raised by our work is whether there is any sense in which an algorithm can be recovered from the categorical presentation of the problem which it solves. For example, do the universal constructions in our factorization admit greedy algorithms? Do they allow us to recover the parallel traversal algorithms of \cite{TK95}? Can we infer that sequential depth-first and breadth-first search require LIFO and FIFO (i.e., stack and queue) data structures respectively?

Positive answers to any of these questions would be strong evidence that the ``categorical structure theory of algorithms'' mentioned in the introduction is actually a viable program; that it can not only describe problems, but suggest ways to solve them---in the elegant characterization of \cite{Abr20}, to \emph{lead} rather than to \emph{follow}.

\bibliographystyle{plainurl}
\bibliography{bibliography}

\begin{thebibliography}{10}

\bibitem{Abr20}
S.~Abramsky.
\newblock Whither semantics?
\newblock {\em Theoretical Computer Science}, 807:3--14, 2020.

\bibitem{ACD20}
E.~Allender, A.~Chauhan, and S.~Datta.
\newblock Depth-first search in directed graphs, revisited.
\newblock {\em Electronic colloquium on computational complexity}, 20(74),
  2020.

\bibitem{CK08}
D.~Corneil and R.~Krueger.
\newblock A unified view of graph searching.
\newblock {\em SIAM J. Discrete Math.}, 22:1259--1276, 01 2008.

\bibitem{TK95}
P.~Delatorre and C.P. Kruskal.
\newblock Fast parallel algorithms for all-sources lexicographic search and
  path-algebra problems.
\newblock {\em Journal of Algorithms}, 19(1):1--24, 1995.

\bibitem{TK01}
P.~Delatorre and C.P. Kruskal.
\newblock Polynomially improved efficiency for fast parallel single-source
  lexicographic depth-first search, breadth-first search, and topological-first
  search.
\newblock {\em Theory of Computing Systems}, 34:275--298, 2001.

\bibitem{DixonKissinger2013}
L.~Dixon and A.~Kissinger.
\newblock Open-graphs and monoidal theories.
\newblock {\em Mathematical Structures in Computer Science}, 23(2):308–359,
  2013.

\bibitem{Kock2016}
J.~Kock.
\newblock Graphs, hypergraphs, and properads.
\newblock {\em Collectanea mathematica}, 67(2):155--190, 2016.

\bibitem{LackSob2005}
S.~Lack and P.~Soboci{\'n}ski.
\newblock Adhesive and quasiadhesive categories.
\newblock {\em {RAIRO-Theoretical Informatics and Applications}},
  39(3):511--545, 2005.

\bibitem{Mas20}
J.~Master.
\newblock The open algebraic path problem, 2020.
\newblock \href {http://arxiv.org/abs/arXiv:2005.06682}
  {\path{arXiv:arXiv:2005.06682}}.

\bibitem{RathSobSte2014}
J.~Rathke, P.~Soboci{\'{n}}ski, and O.~Stephens.
\newblock Compositional reachability in petri nets.
\newblock In J.~Ouaknine, I.~Potapov, and J.~Worrell, editors, {\em
  Reachability Problems}, pages 230--243. Springer, 2014.

\bibitem{RT75}
D.~J. Rose and R.~E. Tarjan.
\newblock Algorithmic aspects of vertex elimination.
\newblock In {\em Proceedings of the Seventh Annual ACM Symposium on Theory of
  Computing (STOC '75)}, pages 245--254. ACM, 1975.

\end{thebibliography}

\newpage

\section*{Appendix: proofs from sections \ref{sec:bg} and \ref{sec:alternate}}

Here we prove the results from these sections as well as some auxiliary definitions and lemmas which are required for the proofs of these results, not their statements.

When reasoning about depth- and breadth-first searches (Definitions \ref{def:depth-first search} and \ref{def:breadth-first search}), the notion of \emph{stages} is important:

\begin{definition}\label{def:stage i}
 Suppose a graph $G$ has $n$ vertices. For a specific computation of depth- or breadth-first traversal over $G$, label the vertices $v_0,\dots,v_{n-1}$ according to the order in which they are added to $L$. (Equivalently, in the order of the traversal defined by this computation.)
 
 For $0 \le i < n$ \emph{stage $i$} refer to the pass of the while loop during which $v_i$ is removed from $S$ or $Q$ and added to $L$. Let $S_i$, respectively $Q_i$, and $L_i$ be the values of $S$, respectively $Q$, and $L$ immediately before stage $i$. Notice that $v_i$ is the top of $S_i$, respectively the front of $Q_i$, and $L_i = (v_0,\dots,v_{i-1})$.
\end{definition}

Even though there are other passes through the while loop besides the labeled stages, no new vertices are added to $S$, $Q$, or $L$ during these.

\begin{proof}[Proof of Lemma \ref{lem:FO characterization of DFT}]
 Suppose $G$ has $n$ vertices, and let $(v_0,v_1,\dots,v_{n-1})$ be the vertices of $G$ according to $\cdot < \cdot$. Fix a computation of depth-first search computing the traversal $<$. 
 
 Consider three vertices $u < v < w$ such that $u \to w$. Let $i$, $j$, and $k$ satisfy $u = v_i$, $v = v_j$ and $w = v_k$; then $0 \le i < j < k \le n-1$. Notice that we push a new occurrence of $w$ onto $S$ at stage $i$,: as $w \in N(v_i)$ and $w \not \in L_i$, $w \in \partial v_i$.

 Notice that $v < w$ in the traversal ordering, which means that by the stage ($j$) that $v$ appears at the top of $S$, $w$ has never appeared at the top of $S$. This means one of two things: either $v$ appears and is discharged from $S$ before $w$ is ever pushed onto it, or $v$ and $w$ co-exist on the stack $S$ simultaneously. 
 
 The first case does not hold: $w$ is pushed onto $S$ at stage $i$, but $v$ has not been discharged from $S$ at this point, or else it would occur in $L_{i+1} = (v_0,\dots,v_i)$. 
 
 Therefore, $v$ and $w$ occur in $S$ simultaneously. Consider all the stages $\ell < j$, before either $v$ or $w$ appears on top of $S$. Since $S$ is a last-in first-out data structure and since $v$ appears on top first, we conclude that for every stage $\ell < j$ at which $w$ gets pushed onto $S$, there is another stage $\ell \le \ell' < j$ at which $v$ gets pushed onto $S$.
 
 In particular, consider $\ell'$ when $\ell = i$, and let $v' = v_{\ell'}$. Then $u \le v'$ (since $i \le \ell'$), $v' < v$ (since $\ell' < j$), and $v' \to v$ (since $v$ is pushed onto $S$ at stage $\ell'$). This is exactly what we wanted to show.
\end{proof}

\begin{proof}[Proof of \ref{lem:FO characterization of BFT}]
 Suppose $G$ has $n$ vertices, and let $(v_0,v_1,\dots,v_{n-1})$ be the vertices of $G$ according to $\cdot < \cdot$. Fix a computation of depth-first search computing the traversal $<$. 
 
 Consider three vertices $u < v < w$ such that $u \to w$. Let $i$, $j$, and $k$ satisfy $u = v_i$, $v = v_j$ and $w = v_k$; then $0 \le i < j < k \le n-1$. Notice that we enqueue a new occurrence of $w$ onto $Q$ at stage $i$,: as $w \in N(v_i)$ and $w \not \in L_i$, $w \in \partial v_i$.

 Notice that $v < w$ in the traversal ordering, which means that by the stage ($j$) that $v$ appears at the front of $Q$, $w$ has never appeared at the front of $Q$. This means one of two things: either $v$ is added to and discharged from $Q$ before $w$ is ever pushed onto it, or $v$ and $w$ co-exist on $Q$ simultaneously.
 
 The first case does not hold: $w$ is enqueued at stage $i$, but $v$ has not been dequeued at this point, or else it would occur in $L_{i+1} = (v_0,\dots,v_i)$. 
 
 Therefore, $v$ and $w$ occur in $S$ simultaneously. Consider all the stages $\ell < j$, before either $v$ or $w$ appears at the front of $Q$. Since $Q$ is a first-in first-out data structure and since $v$ appears in front first, we conclude that for every stage $\ell < j$ at which a copy $w$ is enqueued, there is another stage $\ell' \le \ell$ at which a copy of $v$ is enqueued.
 
 In particular, consider $\ell'$ when $\ell = i$, and let $v' = v_{\ell'}$. Then $v' \le u$ (since $\ell' \le i$) and $v' \to v$ (since $v$ is enqueued at stage $\ell'$). This is exactly what we wanted to show.
\end{proof}

The following is an important algorithmic characterization of $\bfp$ and $\dfp$ which we need in subsequent proofs.

\begin{lemma}\label{lem:characterization of preds}
Suppose $\cdot < \cdot$ is a depth-first traversal, respectively breadth-first traversal, on a finite connected graph $G$ with $n$ vertices. Consider the execution of depth-first search, respectively breadth-first search, computing $\cdot < \cdot$. 
For each such $0 \le i < n$, $v_i$ occurs the top of $S_i$, respectively front of $Q_i$. Let $j < i$ be the stage at which that copy of $v_i$ was added to $S$, respectively $Q$. Then $v_j = \dfp(v_i)$, respectively $v_j = \bfp(v_i)$.
\end{lemma}

\begin{proof}
 Mark every occurrence of $v_i$ that is pushed onto $S$ or $Q$ with the stage at which it was pushed on; hence, the copy above appears as $(v_i,j)$. 
 
 In the depth-first case: suppose there were some $j < j' < i$ such that $v_{j'} \to v_i$; in this case, $(v_i,j')$ is pushed onto $S$ above $(v_i,j)$. But because of the last-in first-out property of $S$, $(v_i,j')$ would appear on top of $S$ before $(v_i,j)$, which contradicts the definition of $(v_i,j)$ as the first occurrence of $v_i$ to appear on top of $S$.
 Hence $v_j$ is the $<$-\emph{greatest} (i.e., latest) preceding neighbor of $v_i$, which is what we wanted to show.
 
In the breadth-first case: suppose there were some $j' < j$ such that $v_{j'} \to v_i$; in this case, $(v_i,j')$ is pushed onto $Q$ before $(v_i,j)$. But because of the first-in first-out property of $S$, $(v_i,j')$ would appear on top of $Q$ before $(v_i,j)$, which contradicts the definition of $(v_i,j)$ as the first occurrence of $v_i$ to appear in the front of $Q$.
 Hence $v_j$ is the $<$-\emph{least} (i.e., earliest)  preceding neighbor of $v_i$, which is what we wanted to show.
\end{proof}

\begin{proof}[Proof of Lemma \ref{lem:bf canonical paths are minimal}]
Let $(v^0,v^1,v^2,\dots,v^{\ell-1})$ be the vertices of $\pi$ in reverse order, where $v^0 = v$ and $v^{\ell-1} = v_0$. Assume by contradiction that $\ell < |\mathfrak{Q}u|$, and let $u^i = \bfp^i(u)$ for $0 \le i \le \ell-1$.

We claim that for all $0 \le i \le \ell-1$, $u^i \lexbfte v^i$. The base case is $u \lexbfte v$. Since $\bfp$ respects $\lexbfte$, $u^i \lexbfte v^i \implies u^{i+1} \lexbfte \bfp(v^i) \lexbfte v^{i+1}$, the second inequality coming from $v^{i+1} \to v_i$.
In particular, $u^{\ell-1} \lexbfte v^{\ell-1} = v_0$, which contradicts $\ell < |\mathfrak{Q}u|$.

For the second claim, assume by contradiction that there exist some $u$ and $v$ and a path $\pi : v_0 \tp v$ such that $u \lexbfte v$, $|\mathfrak{Q}u|=|\pi|$, but $\pi \prec \mathfrak{Q}u$. We may assume that $u$ is the $\lexbft$-least vertex to occur in such a pair, and moreover that neither $u$ nor $v$ is $v_0$. 

Let $u'$ and $v'$ be the second-to-last elements in $\mathfrak{Q}u$ and $\pi$ respectively, and let $\pi' : v_0 \tp v'$ be obtained from $\pi$ by deleting the last element $v$. Then $u' \lexbfte v'$ (as in the proof of the first claim) and $|\mathfrak{Q}u'| = |\pi'|$, so $\mathfrak{Q}u' \preceq \pi'$ by minimality of $u$.

How can it be the case that $\mathfrak{Q}u' \preceq \pi'$ but $(\pi',v) \prec (\mathfrak{Q}u',u)$? It must be the case that $\mathfrak{Q}u' = \pi'$ and $v \lexbft u$, contradicting $u \lexbfte v$.
\end{proof}

\begin{proof}[Proof of Corollary \ref{cor:bf canonical paths are minimal}]
(1) Take $u = v$ and $\pi = \min^s(v_0 \tp v)$ in the statement of the lemma.
(2) In the forward direction, take $\pi = \mathfrak{Q}v$ in the statement of the lemma. In the backwards direction, it suffices to show that $u <_B v \implies \mathfrak{Q}u \prec^s \mathfrak{Q}v$. But, this follows from the forwards direction.
\end{proof}

\begin{proof}[Proof of Lemma \ref{lem:canonical paths are minimal}]
When $v = v_0$, both $\mathfrak{P}v$ and $\min(v_0 \tp v)$ are simply the trivial path $(v_0)$. Therefore, we may assume $v \neq v_0$.

Assume by contradiction that $v$ is the $\lexdft$-least vertex such that $\mathfrak{P}v \neq \min(v_0 \tp v)$. Since these are unequal, co-initial, co-final proper paths, neither is a prefix of the other. Let $w$ be the final vertex on the greatest common prefix of $\mathfrak{P}v$ and $\min(v_0 \tp v)$, let $w \to u_1$ be the next edge of $\mathfrak{P}v$, and let $w \to u_2$ be the next edge of $\min(v_0 \tp v)$. Then $w \to u_2 \edgel w \to u_1$, by minimality of $\min(v_0 \tp v)$.

Since $w = \dfp(u_1)$, it must be the case that $w \lexdft u_1$, and hence $w \lexdft v$. We claim that $w \lexdft u_2$ as well. For if $u_2 \lexdft w$, then $u_2 \lexdft v$; hence $\min(v_0 \tp u_2) = \mathfrak{P}u_2$, and by monotonicity, $\min(v_0 \tp u_2)$ could not contain $w$, which it must. (Notice that $u_2 \neq w$ since $\min(v_0 \tp v)$ is a proper path.)

Consider the stage of depth-first search when $w$ is added to $L$. At this stage, new occurrences of both $u_1$ and $u_2$ are added to $S$, the latter being closer to the top. Since $w  = \dfp(u_1)$, no subsequent occurrence of $u_1$ is pushed onto to $S$ before $u_1$ appears on top of $S$. Therefore, $u_2$ appears on top of $S$ prior to $u_1$, which means $w \lexdft u_2 \lexdft u_1$. In particular, since $u_1 \le v$, $u_2 \neq v$.

We finish the proof by a case analysis. 

\begin{itemize}
    \item Suppose that for each vertex $x \in \min(v_0 \tp v)$, $x \le u_1$. Then, since $u_1 \le v$ and $v \in \min(v_0 \tp v)$, it must be the case that $u_1 = v$. Let $z$ be the penultimate vertex in $\min(v_0 \tp v)$. Then there is an edge $z \to v$, and by the case assumption, $z \lexdft v$. Since $w = \dfp(v)$, $z \le w$, and by minimality of $v$, $ \mathfrak{P}z = \min(v_0 \tp z)$. In particular, the latter path is monotone in $\lexdft$; therefore, all of its elements are $\le z$. However $u_2 \in \min(v_0 \tp z)$ and $z \le w \lexdft u_2$, contradiction.
    
    \item Suppose that there exists a vertex $x \in \min(v_0 \tp v)$ such that $x > u_1$. Let $x$ be the first such vertex in $\min(v_0 \tp v)$. Notice that in particular, $x$ must come after $u_2$, since $u_2 \lexdft v \implies \min(v_0 \tp u_2) = \mathfrak{P}u_2$ is monotone in $\lexdft$, and hence cannot contain $x$.
    
    Let $y$ be the vertex immediately preceding $x$ in $\min(v_0 \tp v)$, so that $y \le u_1$. In fact, since $\min(v_0 \tp y)$ contains $u_2$ and is monotone in $\lexdft$ (being identical to $\mathfrak{P}y)$), it must be the cause that $u_2 \le y \le u_1$, in particular, $w \lexdft y$.
    
    If $y \neq u_1$, then $y \lexdft u_1 \lexdft x$, $y \to x$, but $\dfp(u_1) \lexdft y$, contradicting the depth-first traversal property of Lemma \ref{lem:FO characterization of DFT}.
    
    If $y = u_1$, then let $z$ be the vertex immediately preceding $y$ in $\min(v_0 \tp v)$. For the same reason as $y$, $u_2 \le z$ ($\min(v_0 \tp z)$ must contain $u_2$ and is monotone in $\lexdft$). Hence $w \lexdft z \lexdft y = u_1$, but $z \to u_1$, contradicting $w = \dfp(u_1)$
\end{itemize}

This concludes the proof.
\end{proof}

\begin{lemma}\label{lem:prec_D has the DFT property}
For any $u \prec_D v \prec_D w$, if $u \to w$, then there exists some $u \preceq_D v' \prec_D v$ such that $v' \to v$.
\end{lemma}

\begin{proof}
Let $\mathbf{u}$, $\mathbf{v}$, and $\mathbf{v}$ refer to $\min(v_0 \tp u)$, $\min(v_0 \tp v)$, and $\min(v_0 \tp w)$ respectively. By assumption $\mathbf{u} \prec \mathbf{v} \prec \mathbf{w}$, and since $u \to w$, $\mathbf{w} \preceq (\mathbf{u},w)$. In particular, $\mathbf{u} \prec \mathbf{v} \prec (\mathbf{u},w)$.

We claim that $\mathbf{u} \sqsubseteq \mathbf{v}$, i.e., that $\mathbf{u}$ is a prefix of $\mathbf{v}$. Otherwise, if the longest common prefix of the two paths were a proper prefix of each, then the same holds of $(\mathbf{u},w)$ and $\mathbf{v}$, and moreover, $\mathbf{u} \prec \mathbf{v} \implies (\mathbf{u},w) \prec \mathbf{v}$, a contradiction.

Therefore, $\mathbf{u} \sqsubseteq \mathbf{v}$; indeed, $\mathbf{u} \sqsubset \mathbf{v}$ as they are not equal. Let $v'$ be the penultimate vertex in $\mathbf{v}$; then $v' \to v$, and $\mathbf{v} = (\min(v_0 \tp v'),v)$, since minimal paths are closed under taking prefixes. In particular, $\mathbf{u} \sqsubseteq \min(v_0 \tp v')$, hence $u \preceq_D v'$, which is what we needed to show. 
\end{proof}

\begin{proof}[Proof of Theorem \ref{thm:alternate-formulation}]
Recall the definition of \emph{stage $i$} (Definition \ref{def:stage i}), for $0 \le i < n$, the vertex $v_i$, and the values $L_i$ and $S_i$ of $L$ and $S$.
In this notation, $i < j \iff v_i \lexdft v_j$, so we can restate our theorem as
$ v_i \prec_D v_j \iff i < j,$
for every $(i,j) \in n^2$.

Let $\vec{u} = (u_0,\dots,u_{n-1})$ be the vertices of $G$ ordered according to 
$\cdot \prec_D \cdot$, and suppose by contradiction that $\vec{u} \neq \vec{v}$, 
where $\vec{v} = (v_0,\dots,v_{n-1})$. Let $\hat{\imath}$ be the least index $i$ such
that $u_i \neq v_i$. Since $v_0$ is the $\prec_D$-least vertex of $G$, $u_0 = v_0$
and $\hat{\imath} > 0$. By Lemma \ref{lem:FO characterization of DFT}, $u_i$ and 
$v_i$ have preceding neighbors in $\vec{u}$ and $\vec{v}$. Let $\hat{\jmath}$ and 
$\hat{\jmath}'$ be the \emph{greatest} indices less than $i$ such that there exist 
edges $u_{\hat{\jmath}} \to u_{\hat{\imath}}$ and $v_{\hat{\jmath}'} \to 
v_{\hat{\imath}}$ respectively.
By minimality of $\hat{\imath}$, $u_{\hat{\jmath}} = v_{\hat{\jmath}}$ and 
$u_{\hat{\jmath}'} = v_{\hat{\jmath}'}$. Moreover, there exist $\hat{k}$, $\hat{k}' >
i$ such that $v_i = u_{\hat{k}}$ and $u_i = v_{\hat{k}'}$. Therefore, 
$v_{\hat{\jmath}} \to v_{\hat{k}'}$ and $u_{\hat{\jmath}'} \to u_{\hat{k}}$.

Suppose that $\hat{\jmath} < \hat{\jmath}'$. Then we have $\hat{\jmath} < \hat{\jmath}' < \hat{\imath} < \hat{k}$, but 
$ u_{\hat{\jmath}} \to u_{\hat{\imath}}$ and $ u_{\hat{\jmath}'} \to u_{\hat{k}}  $, contradicting the DFT property of $\prec_D$ (Lemma \ref{lem:prec_D has the DFT property}). Similarly if $\hat{\jmath}' < \hat{\jmath}$, then $\hat{\jmath}' < \hat{\jmath} < \hat{\imath} < \hat{k}'$, but 
$ v_{\hat{\jmath}'} \to v_{\hat{\imath}}$ and $ v_{\hat{\jmath}} \to u_{\hat{k}'}  $, contradicting the DFT property of $\lexdft$ (Lemma \ref{lem:FO characterization of DFT}). Therefore $\hat{\jmath} = \hat{\jmath}'$, and there is an edge from $v_{\hat{\jmath}} = u_{\hat{\jmath}'}$ to both $v_{\hat{\imath}}$ and $v_{\hat{k}'} = u_{\hat{\imath}}$.

By definition, $v_{\hat{\jmath}}$ is the top of the stack $S_{\hat{\jmath}}$. Therefore, $S_{\hat{\jmath}+1}$ is obtained from $S_{\hat{\jmath}}$ by popping $v_{\hat{\jmath}}$ off and pushing all of its neighbors on, in reverse neighborhood order, which are not contained in $L_{\hat{\jmath}+1}$; this includes both $v_{\hat{\imath}}$ and $u_{\hat{\imath}}$. We claim that this is the \emph{final} occurrence of $v_{\hat{\imath}}$ that is pushed onto $S$.

Otherwise, if there were some $\ell > \hat{\jmath}$ such that $S_{\ell + 1}$ contains
a new occurrence of $v_{\hat{\imath}}$, then it must be the case that $\ell < 
\hat{\imath}$, for otherwise $v_{\hat{\imath}} \in L_{\ell+1}$, which would prevent 
it from being pushed onto $S_{\ell+1}$. But it also must be the case that there is an
edge $v_{\ell} \to v_{\hat{\imath}}$, contradicting the maximality of 
$v_{\hat{\jmath}}$ among prior neighbors of $v_{\hat{\imath}}$. 

Since this is the final occurrence of $v_{\hat{\imath}}$ that is pushed on $S$, and since $v_{\hat{\imath}}$ is popped off of $S$ \emph{before} $u_{\hat{\imath}}$, we can infer that (i) $u_{\hat{\imath}}$ is pushed onto $S_{\hat{\jmath}+1}$ before $v_{\hat{\imath}}$, and (ii) no subsequent occurrence of $u_{\hat{\imath}}$ is pushed onto $S$. This means that (i) $v_{\hat{\jmath}} \to v_{\hat{\imath}} \edgel v_{\hat{\jmath}} \to u_{\hat{\imath}}$, and (ii) $\dfp(u_{\hat{\imath}})= v_{\hat{\jmath}}$.

Since $\dfp(v_{\hat{\imath}}) = v_{\hat{\jmath}}$ as well, we can infer that $\mathfrak{P}u_{\hat{\imath}} = \mathfrak{P}v_{\hat{\jmath}} \to u_{\hat{\imath}}$ and $\mathfrak{P}v_{\hat{\imath}} = \mathfrak{P}v_{\hat{\jmath}} \to v_{\hat{\imath}}$. By Lemma \ref{lem:canonical paths are minimal}, $\min(v_0 \tp u_{\hat{\imath}}) = \min(v_0 \tp v_{\hat{\jmath}}) \to u_{\hat{\imath}}$ and $\min(v_0 \tp v_{\hat{\imath}}) = \min(v_0 \tp v_{\hat{\jmath}}) \to v_{\hat{\imath}}$.

Since $u_{\hat{\imath}} \prec_D v_{\hat{\imath}}$, $\min(v_0 \tp u_{\hat{\imath}}) \prec \min(v_0 \tp v_{\hat{\imath}})$, which means in particular that $v_{\hat{\jmath}} \to u_{\hat{\imath}} \edgel v_{\hat{\jmath}} \to v_{\hat{\imath}}$. But this contradicts assertion (i) above.

\end{proof}

\section*{Appendix: proofs from sections \ref{sec:graphs} through \ref{sec:putting}}

\begin{proof}[Proof of Lemma \ref{lem:EO homs preserve lex order}]
 It suffices to show that $h$ preserves $\cdot \prec \cdot$; the second statement follows from the observation that graph homomorphisms preserve path length, i.e., $|\pi| = |h(\pi)|$ for every path $\pi$.
 
 Suppose that $\pi$ and $\sigma$ are co-initial paths in $G$ such that $\pi \prec \sigma$. If $\pi \sqsubset \sigma$, then $h(\pi) \sqsubset h(\sigma)$ and we're done. Otherwise, let $u \to v_1$ and $u \to v_2$ be the first edges in $\pi$ and $\sigma$ respectively following their longest common prefix $\zeta$. Since $\pi \prec \sigma$, $u \to v_1 \edgel u \to v_2$, so $h(u) \to h(v_1) \edgel h(u) \to h(v_2)$.
 
 Notice that $h(\zeta)$ is a common prefix of $h(\pi)$ and $h(\sigma)$, and since $h(u) \to h(v_1) \edgel h(u) \to h(v_2)$, it must be the longest one. Moreover, since $h(u) \to h(v_1) \edgel h(u) \to h(v_2)$, $h(\pi) \prec h(\sigma)$, which is what we wanted to show.
\end{proof}

\begin{lemma}\label{lem:theta well defined}
The functor $\T$ is well-defined.
\end{lemma}

\begin{proof}
We first have to check that $\T(G)$ is a finite, edge-ordered arborescence. It is trivially finite and edge-oriented. Notice that $\T(G)$ is connected: for every vertex $u$, every edge on the path $\min(v_0 \tp u)$ is contained in $\T(G)$. On the other hand, the in-degree of each vertex is at most 1: there cannot be two distinct edges $u \to v$ and $u' \to v$ in $\min(v_0 \tp v)$.

Next we have to check that for any homomorphism $h$, $\T(h)$ is in fact a homomorphism of pointed, edge-ordered graphs. First notice that $\T(h)$ actually maps into $\T(H)$: if $u \to v$ is included in $\T(G)$, then it must be contained in $\min(v_0 \tp v)$ in $G$, whose $h$-image is $\min(h(v_0) \tp h(v))$ by property \eqref{def:EO_hom:min} of Definition \ref{def:EO_hom}; therefore, $h(u \to v)$ is contained in a least path and so included in $\T(H)$.

Moreover, $\T(h)$ clearly fixes the distinguished vertex, and visibly inherits monotonicity on co-initial edges from $h$. Hence, $\T(h)$ is a pointed, edge-ordered graph homomorphism.
\end{proof}

\begin{proof}[Proof of Theorem \ref{thm:adj graphs with arbs (lex)}]
First, given $\I(T) \tot{h^\abv} G$ in $\mathbf{FinGraph}_\star^l$, we describe how to obtain $T \tot{h_\blw} T(G)$ in $\mathbf{FinArb}_\star^<$. Define $h_\blw(v) = h^\abv(v)$, for $v \in T$ (which, remember is the same as $I(T)$).
To check that $h_\blw$ is well-defined, we must show that if $u \to v$ is an edge of $T$, then $h^\abv(u \to v)$ is contained in $\T(G)$. But since $T$ is an arborescence, $u \to v$ is trivially contained in $\min(v_0 \tp v)$, so $h^\abv(u \to v)$ is contained in $\min(h(v_0) \tp h(v))$, and hence included in $\T(G)$.

Moreover, the fact that $h_\blw$ maps edges to edges, preserves the distinguished point, and is monotone on co-initial edges, is immediately inherited from $h^\abv$.

Next, given $T \tot{h^\blw} \T(G)$ in $\mathbf{FinArb}^<_\star$, define $I(T) \tot{h^\abv} G$ in $\mathbf{FinGraph}_\star^l$ by $h^\abv(v) = h_\blw(v)$, for $v \in \I(T)$.
In this case, we do not need to check that $h^\abv$ is well-defined, and the properties of mapping edges to edges, preserving the distinguished point, and monotonicity on neighborhoods, are inherited immediately from $h_\blw$. The fact that $h^\abv$ maps minimal paths to minimal paths is easily justified by observing that $h^\abv$ maps into $\T(G)$, the tree of minimal paths in $G$.

To check that this correspondence is bijective, notice that starting from either $h^\abv$ or $h_\blw$, passing to the other one, and passing back again, gives us the same morphism we started with. Naturality follows straightforwardly by observing that both functors are the identity on morphisms, so naturality squares trivially commute.
\end{proof}

\begin{lemma}\label{lem:adjunction for free}
There is an adjunction
\[\begin{tikzcd}
	{\mathbf{LexGraph}} && {\mathbf{FinArb}^<_\star}
	\arrow["{\T'}"{name=0}, from=1-1, to=1-3, curve={height=-18pt}]
	\arrow["{\I'}"{name=1}, from=1-3, to=1-1, curve={height=-18pt}]
	\arrow["\dashv"{rotate=90}, from=1, to=0, phantom]
\end{tikzcd}\]
\end{lemma}

\begin{proof}
The adjunction from Theorem~\ref{thm:adj graphs with arbs (lex)} factors as
\[\begin{tikzcd}
	{\mathbf{LexGraph}} & {\mathbf{FinGraph}^l_\star} \\
	& {\mathbf{FinArb}^<_\star}
	\arrow["{\Theta}"{name=0}, from=1-2, to=2-2, curve={height=-12pt}]
	\arrow["{I}"{name=1}, from=2-2, to=1-2, curve={height=-12pt}]
	\arrow["{J}", from=1-1, to=1-2, tail]
	\arrow["{I'}", from=2-2, to=1-1, curve={height=-12pt}]
	\arrow["\dashv"{rotate=0}, from=1, to=0, phantom]
\end{tikzcd}\]
where $J$ is the fully faithful identity-on-objects functor witnessing the inclusion of $\mathbf{LexGraph}$ in $\mathbf{FinGraph}^l_\star$, and $\Theta' = \Theta \circ J$. The natural isomorphism 
\begin{align*}
\mathbf{LexGraph}(I'(G),H) & \cong \mathbf{FinGraph}_\star^l(J(I'(G)),J(H)) \\
& = \mathbf{FinGraph}_\star^l(I(G), J(H)) \\
& \cong \mathbf{FinArb}_\star^<(G, \Theta(J(H))) \\
& = \mathbf{FinArb}_\star^<(G, \Theta'(H))
\end{align*}
establishes this adjunction.
\end{proof}

\begin{lemma}\label{lem:S well defined}
For any finite, connected, pointed, edge-ordered graph $G$, $S(G)$ is an arborescence.
\end{lemma}

\begin{proof}
To verify that $S(G)$ is a well-defined arborescence, it suffices to observe that $S(G)$ is connected (as it preserves least shortest paths), and that there is a unique path $v_0 \tp v$: two distinct edges $u \to v$ and $u' \to v$ cannot both occur in $\min^s(v_0 \tp v)$.

To check that $S(G \tot{h} H)$ is well defined, we have to verify for $u \to v \in S(G)$, $h(u) \to h(v) \in S(H)$. But such morphisms $h$ preserve least shortest paths by Definition \ref{def:EO_hom} (\ref{def:EO_hom:shortmin}).
\end{proof}

\begin{proof}[Proof of Theorem \ref{thm:adj graphs with arbs (short-lex)}]
The proof is obtained by literally copying the proof of Theorem \ref{thm:adj graphs with arbs (lex)} and replacing $\mathbf{FinGraph}^l_\star$ by $\mathbf{FinGraph}^s_\star$, $\T$ by $S$, and $\min$ by $\min^s$ throughout.

Given $\I(T) \tot{h^\abv} G$ in $\mathbf{FinGraph}_\star^s$, we describe how to obtain $T \tot{h_\blw} T(G)$ in $\mathbf{FinArb}_\star^<$. Define $h_\blw(v) = h^\abv(v)$, for $v \in T$ (which we identified with $\I(T)$).
To check that $h_\blw$ is well-defined, we must show that if $u \to v$ is an edge of $T$, then $h^\abv(u \to v)$ is contained in $S(G)$. But since $T$ is an arborescence, $u \to v$ is trivially contained in $\min^s(v_0 \tp v)$, so $h^\abv(u \to v)$ is contained in $\min^s(h(v_0) \tp h(v))$, and hence included in $S(G)$.

Moreover, the fact that $h_\blw$ maps edges to edges, preserves the distinguished point, and is monotone on co-initial edges, is immediately inherited from $h^\abv$.

Next, given $T \tot{h^\blw} S(G)$ in $\mathbf{FinArb}^<_\star$, define $I(T) \tot{h^\abv} G$ in $\mathbf{FinGraph}_\star^s$ by
$h^\abv(v) = h_\blw(v)$, for $v \in \I(T)$ (which we identified with $T$).
In this case, we do not need to check that $h^\abv$ is well-defined, and the properties of mapping edges to edges, preserving the distinguished point, and monotonicity on neighborhoods, are inherited immediately from $h_\blw$. The fact that $h^\abv$ maps minimal shortest paths to minimal shortest paths is easily justified by observing that $h^\abv$ maps into $S(G)$, the tree of minimal shortest paths in $G$.

To check that this correspondence is bijective, notice that starting from either $h^\abv$ or $h_\blw$, passing to the other one, and passing back again, gives us the same morphism we started with. Naturality follows straightforwardly by observing that both functors are the identity on morphisms, so naturality squares trivially commute.
\end{proof}

\begin{proof}[Proof of Lemma \ref{lem:min}]
  (i): If there are edges $u \to v_1$ and $u \to v_2$ of $G$, then $u \to v_1 \edgel u \to v_2$ in $G$ iff $\min(u \tp v_1) \prec \min(u \tp v_2)$ in $G$ (since $G$ is a lex-graph) iff $u \to v_1 \edgel u \to v_2$ in $\F(G)$ (by definition of $\F$).
  
  (ii): Suppose $\sigma$ and $\pi$ are co-initial paths from $G$. We may assume that $\sigma$ and $\pi$ share no nontrivial prefix; then either $\sigma$ or $\pi$ is empty (and the claim is trivial), or $\sigma$ and $\pi$ differ on their first edge, and the claim follows from (i).

  (iii): Work in $\F(G)$. It suffices to show that minimal paths in $\F(G)$ consist of only edges in $G$; by the above remark, if two paths consist of $G$-edges, it does not matter whether we compare them in $G$ or in $\F(G)$.
  
  Towards which, it suffices to show that every edge $u \to v$ not in $G$ is greater than the minimal path $\min_G(u \tp v)$ between $u$ and $v$ in $G$. Then any path with non-edges in $G$ can be lessened; in particular minimal paths in $\F(G)$ cannot contain any non-edges of $G$.
  
  Let $v_1 \neq v$ be the second vertex in $\min_G(u \tp v)$. Since minimal paths are closed under taking prefixes, $u \to v_1 = \min_G(u \tp v_1)$, in both $G$ and $\F(G)$. Since $\min_G(u \tp v_1) \prec \min_G(u \tp v)$, $u \to v_1 \edgel u \to v$ in $\F(G)$. But then $\min_G(u \tp v) \prec u \to v$, which is what we wanted to show.
\end{proof}

\begin{proof}[Proof of Lemma \ref{lem:free_functor}]
  Given any lex-graph $G$, $\F(G)$ is clearly a transitive edge-ordered graph, but we must check that it satisfies the lex-graph property (Definition \ref{def:lex_graph}-(\ref{def:lex_graph:order})). But for any $v_1$ and $v_2$ in the neighborhood of a common $u$ in $\F(G)$, $u \to v_1 \edgel u \to v_2$ in $\F(G)$ iff $\min(u \tp v_1)
  \prec \min(u \tp v_2)$ in $G$ (by definition of $\F$) iff $\min(u \tp v_1)
  \prec \min(u \tp v_2)$ in $\F(G)$ (by Lemma \ref{lem:min} plus the preceding remark).

  By definition, $\F$ immediately preserves identities and compositions of homomorphisms, and it remains to check that for any lex-homomorphism $G \tot{h} H$ of lex-graphs, $\F(h)$ is a lex-homomorphism $\F(G) \to \F(H)$. We check each of the conditions in Definition~\ref{def:EO_hom}:
  \begin{enumerate}[(i)]
      \item If $u \to v$ in $\F(G)$, then there is a path $u \tp v$ in $G$, so there is a path $h(u) \tp h(v)$ in $H$, and hence an edge $h(u) \to h(v)$ in $\F(H)$.
      
      \item Since $h$ preserves the distinguished point, so does $\F(h)$.
      
      \item If $u \to v_1 \edgel u \to v_2$ in $G$, $\min(u \tp v_1) \prec \min(u \tp v_2)$, as $G$ is a lex-graph. Since $h$ is a lex-homomorphism, $\min_H(h(u) \tp h(v_i)) = h(\min_G(u \tp v_i))$, so $\min(h(u) \tp h(v_1)) \prec \min(h(u) \tp h(v_2))$ in $H$. By definition of $\F$, $h(u) \to h(v_1) \edgel h(u) \to h(v_2)$ in $\F(H)$.
      
      \item As the minimal path $\min(u \tp v)$ in $\F(G)$ is also minimal in $G$, as shown above, and as $h$ preserves minimal paths, $h(\min(u \tp v))$ is $\min(h(u) \tp h(v))$ in $H$, and this in turn is also the minimal path in $\F(H)$.
  \end{enumerate}
\end{proof}

\begin{proof}[Proof of Theorem \ref{thm:lex/tlex adjunction}]
  Let $\F(G) \tot{h^\abv} H$ be a homomorphism of
  $\mathbf{TLexGraph}$. Since $\F(G)$ has the same vertices as $G$, we may
  define $G \tot{h_\blw} \U(H)$ by $h_\downarrow(v) = h^\uparrow(v)$
  (since $H$ and $\U(H)$ are exactly identical).
  
  We check that $h_\blw$ is a lex-homomorphism by checking Definition \ref{def:EO_hom} (\ref{def:EO_hom:edge})-(\ref{def:EO_hom:min}). We write, e.g., $h(v)$ to refer unambiguously to the vertex $h^\abv(v) = h_\blw(v)$.
  \begin{enumerate}[(i)]
      \item If $u \to v$ is an edge of $G$, it is an edge of $\F(G)$, so $h(u) \to h(v)$ is an edge of $H$, and hence an edge of $U(H)$.
      
      \item $h_\blw$ directly inherits preservation of the distinguished point from $h^\abv$
      
      \item If $u \to v_1 \edgel u \to v_2$ in $G$, then $u \to v_1 \edgel u \to v_2$ in $\F(G)$ (Lemma \ref{lem:min}), so $h(u) \to h(v_1) \edgel h(u) \to h(v_2)$ in $H$ (monotonicity of $h^\abv$), and hence the same holds in $U(H)$.
      
      \item If $\pi$ is the minimal path $u \tp v$ in $G$, then it's minimal in $\F(G)$ (Lemma \ref{lem:min}), so $h(\pi)$ is minimal in $H$ (since $h^\abv$ is a lex-homomorphism), and hence also in $U(H)$.
  \end{enumerate}
  
  In the other direction, we suppose that we are given some homomorphism of 
  lex-graphs $G \tot{h_\blw} \U(H)$ in $\mathbf{LexGraph}$. Define $\F(G) \tot{h^\abv} H$ by 
  $h^\abv(v) = h_\blw(v)$. Again, we check (\ref{def:EO_hom:edge})-(\ref{def:EO_hom:min}) of Definition \ref{def:EO_hom}.
  \begin{enumerate}[(i)]
      \item If $u \to v$ is an edge of $\F(G)$, then there is a path $u \tp v$ in $G$, and hence a path $h(u) \tp h(v)$ in $U(H)$, and (since $H$ is transitive), and edge $h(u) \to h(v)$.
      
      \item $h^\abv$ directly inherits preservation of the distinguished point from $h_\blw$
      
      \item If $u \to v_1 \edgel u \to v_2$ in $\F(G)$, then $\min(u \tp v_1) \prec \min(u \tp v_2)$ in $G$; hence (since $h_\blw$ is a lex-homomorphism) $\min(h(u) \tp h(v_1)) \prec \min(h(u) \tp h(v_2))$ in $U(H)$, hence in $H$. Since $H$ is a lex-graph, $h(u) \to h(v_1) \edgel h(u) \to h(v_2)$.
      
      \item If $\pi$ is the minimal path $u \tp v$ in $\F(G)$, then it's the minimal path in $G$ (Lemma \ref{lem:min}), so $h(\pi)$ is minimal in $U(H)$ (since $h_\blw$ is a lex-homomorphism), and hence in $H$. 
  \end{enumerate}
 
  We need only now to check that this correspondence is bijective and natural. Bijectivity follows readily from the fact that we always have $h^\abv = h_\blw$, so going from $G \tot{h_\blw} \U(H)$ to $\F(G) \tot{h^\abv} H$ and back has no 
  effect, and similarly in the other direction. Similarly, naturality follows straightforwardly by noting that $\U(h)(v) = \F(h)(v) = h(v)$ for all $h$ and $v$, so naturality squares trivially commute.
\end{proof}

\begin{proof}[Proof of Lemma \ref{lem:longest paths are tree edges}]
Notice that distances between vertices are never increased in $G$ compared to $T$, only decreased, meaning that if there is a path $u \tp v$ in $G$, there is a path $u \tp v$ in $T$ that is no shorter. Therefore, longest paths in $G$ must consist entirely of edges in $T$, and are therefore unique.

Conversely, if $u \to v$ is an edge of $T$, then it appears on the unique path $v_0 \tp v$ in $T$. As just observed, the longest path $v_0 \tp v$ in $G$ is also a path in $T$; hence, it is the unique path $v_0 \tp v$ in $T$, and thus contains $u \to v$.
\end{proof}

\begin{lemma}\label{lem:Gamma well defined}
$\G$ is a well-defined functor.
\end{lemma}
\begin{proof}
Clearly $\G$ preserves identities and composition. We must check that for any morphism $h$ of $\mathbf{FinArb}^<_\star$, $\G(h)$ satisfies Definition \ref{def:EO_hom} (\ref{def:EO_hom:edge})-(\ref{def:EO_hom:monotone}) and preserves longest paths:
\begin{itemize}
    \item (\ref{def:EO_hom:edge}) If $u \to v$ in $\G(T)$, then $u \tp v$ in $T$, so $h(u) \tp h(v)$ in $T'$, so $h(u) \to h(v)$ in $\G(T')$.
    
    \item (\ref{def:EO_hom:point}) Since $h$ maps the distinguished point, and only that point, of $T$ to the distinguished point of $T'$, $\G$ does the same from $\G(T)$ to $\G(T')$.
    
    \item (\ref{def:EO_hom:monotone}) If $u \to v_1 \edgel u \to v_2$ in $\G(T)$, then $v_0 \tp v_1 \prec^s v_0 \tp v_2$ in $T$. By Lemma \ref{lem:EO homs preserve lex order}, $h(v_0 \tp v_1) \prec^s h(v_0 \tp v_2)$ in $T'$ , and since paths in arborescences are unique, $h(u) \tp h(v_1) \prec^s h(u) \tp h(v_2)$ in $T'$. Hence $h(u) \to h(v_1) \edgel h(v) \to h(v_2)$ in $\G(T')$. 
    
    \item Finally, if $u \tp v$ is the longest path in $\G(T)$, then it is a path in $T$ by Lemma \ref{lem:longest paths are tree edges}, and therefore, $h(u \tp v) = h(u) \tp h(v)$ is a path in $T'$. Again by Lemma \ref{lem:longest paths are tree edges}, $h(u \tp v)$ is the longest path in $\G(T')$.
\end{itemize}
\end{proof}

\begin{lemma}\label{lem:L is well-defined}
$\Ell$ is a well-defined functor $\mathbf{TArb} \to \mathbf{FinArb}^<_\star$.
\end{lemma}

\begin{proof}
In each graph $G \in \mathbf{TArb}$, the unique longest paths are closed under taking prefixes. Therefore, the union of all least longest paths forms an arborescence.

To check that $\Ell$ is a functor, note that for any morphism $h \in \mathbf{TArb}$, $\Ell(h)$ clearly preserves the distinguished point and is monotone on neighborhoods. We only need to show that if $u \to v$ is an edge in $\Ell(G)$ and $h : G \to H$ is a morphism in $\mathbf{TArb}$, then $h(u) \to h(v)$ is an edge of $\Ell(H)$. But, this guaranteed by the fact that $h$ preserves longest paths.

Finally, note that $\Ell$ preserves the identity morphism and respects composition.
\end{proof}

\begin{proof}[Proof of Theorem \ref{thm:arb/TArb adjunction}]
Fix a pointed, connected, edge-ordered arborescence $T$ and a pointed, connected, transitive, edge-ordered graph $G$.

Given $\G(T) \tot{h^\abv} G$ in $\mathbf{TArb}$, we define $T \tot{h_\blw} \Ell(G)$ by $h_\blw(v) = h^\abv(v)$. This is well-defined, because if $u \to v$ is an edge in $T$, then by Lemma \ref{lem:longest paths are tree edges}, it appears on the unique longest path $v_0 \tp v$ in $\G(T)$. Since $h^\abv$ preserves least longest paths, the edge $u \to v$ maps into $\Ell(G)$.  Moreover, $h_\blw$ preserves the distinguished point and inherits monotonicity in neighborhoods from $h^\abv$, so satisfies the conditions of Definition \ref{def:EO_hom} and is a morphism in $\mathbf{FinArb}^<_\star$.

In the other direction, given $T \tot{h_\blw} \Ell(G)$ in $\mathbf{FinArb}^<_\star$, we define $\G(T) \tot{h^\abv} G$ by $h^\abv(v) = h_\blw(v)$; let us unambiguously write $h(v)$ for brevity. If $u \to v$ is an edge of $\G(T)$, then there is a path $u \tp v$ in $T$, hence a path $h(u) \tp h(v)$ in $\Ell(G)$, and therefore an edge $h(u) \to h(v)$ in $G$. Moreover, if $u \to v_1 \edgel u \to v_2$ in $\G(T)$, then $v_0 \tp v_1 \prec^s v_0 \tp v_2$ in $T$, so $h(u) \tp h(v_1) \prec^s h(u) \tp h(v_2)$ in $\Ell(G)$, by Lemma \ref{lem:EO homs preserve lex order}. By the remark succeeding Lemma \ref{lem:longest paths are tree edges}, $h(u) \to h(v_1) \edgel h(u) \to h(v_2)$ in $G$. Finally, if $u \tp v$ is the longest path from $u$ to $v$ in $\G(T)$, then each of its edges lies in $T$ by Lemma \ref{lem:longest paths are tree edges}, hence its $h$-image lies in $\Ell(G)$, which means it is a longest path of $G$. Therefore, $h^\abv$ is a morphism of $\mathbf{TArb}$.

It remains to show that the maps relating $h^\abv$ and $h_\blw$ are bijective and natural. As in the proof of Theorem \ref{thm:lex/tlex adjunction}, this is immediate from the definition of each map; the only thing to show is that they were well-defined.
\end{proof}

\begin{proof}[Proof of Lemma \ref{lem:correctness of construction of DFT}]
Fix $u,v \neq v_0$. By definition of $\F$, $v_0 \to u \edgel v_0 \to v$ in $T$ iff $\min(v_0 \tp u) \prec \min(v_0 \tp v)$ in $(\I' \circ \T)(G)$. Since $\I'$ is an inclusion functor, this is equivalent to $\min(v_0 \tp u) \prec \min(v_0 \tp v)$ in $\T(G)$; indeed, the \emph{unique} path $v_0 \tp u$ is less than the \emph{unique} path $v_0 \tp v$ in $\T(G)$.

But the unique paths $v_0 \tp u$ and $v_0 \tp v$ in $\T(G)$ are exactly $\min(v_0 \tp u)$ and $\min(v_0 \tp v)$ in $G$ respectively; moreover, the relative order on the latter two paths in $G$ is inherited from the relative order on the former two in $\T(G)$.

Therefore, $v_0 \to u \edgel v_0 \to v$ in $T$ iff $\min(v_0 \tp u) \prec \min(v_0 \tp v)$ in $G$, but this is exactly the relation $\lexdft$ of Definition \ref{def:prec D}, which is the lexicographic depth-first traversal of $G$ by Theorem \ref{thm:alternate-formulation}.
\end{proof}

\begin{proof}[Proof of \ref{lem:correctness of construction of BFT}]
Fix $u,v \neq v_0$. By definition of $\G$, $v_0 \to u \edgel v_0 \to v$ in $T$ iff $v_0 \tp u \prec^s v_0 \tp v$ in $S(G)$ (where paths from $v_0$ are unique). By definition of $S$, this is equivalent to $\min^s(v_0 \tp u) \prec^s \min^s(v_0 \tp v)$ in $G$. By Corollary \ref{cor:bf canonical paths are minimal}, this is equivalent to $u \lexbft v$.
\end{proof}

\begin{proof}[Proof of Lemma \ref{lem:fingraph to finloset}]
For a finite, edge-ordered graph $G$, we define $E(G)$ to be the order $(\{v_0\} \cup N(v_0), <)$, where for $u,v \neq v_0$, $u < v \iff v_0 \to u \edgel v_0 \to v$, and for $u \neq v_0$, $v_0 < u$.

On morphisms, given a homomorphism of edge-ordered graphs $G \tot{h} H$, we define $\E(G) \tot{\E(h)} \E(H)$ by $\E(h)(v) = h(v)$. Notice that $E$ is well-defined, since it maps the distinguished point $v_0$ of $G$ to the distinguished point $w_0$ of $H$, and also maps $N_G(v_0)$ into $N_H(w_0)$. By definition, it is clear that $\E$ preserves both identities and compositions, so we have only left to show that $\E(h)$ is monotone.

Since $h$ maps only $v_0$ to $w_0$, it suffices to show that if $u,v \in N_G(v_0)$ and $v_0 \to u \edgel v_0 \to v$ in $G$, then $w_0 \to h(u) \edgel w_0 \to h(v)$ in $H$. But this follows from monotonicity of $h$ (Definition \ref{def:EO_hom}-(\ref{def:EO_hom:monotone}))
\end{proof}

\end{document}